\documentclass[a4paper,11pt]{amsart}

\usepackage{amsmath, amssymb, amsthm}
\usepackage{graphicx}
\usepackage{subfigure}
\usepackage{color}
\usepackage[hyperindex,breaklinks]{hyperref}
\usepackage{xspace}
\usepackage{framed}
\usepackage{algorithm}
\usepackage[noend]{algpseudocode}

\usepackage[a4paper]{geometry} \newtheorem{theorem}{Theorem}

\newtheorem{corollary}{Corollary}

\newtheorem{example}{Example}
\newtheorem{lemma}[theorem]{Lemma}

\newtheorem{proposition}[theorem]{Proposition}
\newtheorem{remark}[theorem]{Remark}
\newenvironment{acknowledgements} {\begin{abstract}} {\end{abstract}}

\hypersetup{colorlinks = true, urlcolor = blue, linkcolor = blue,
  citecolor = blue }

\newcommand{\xmath}[1]{\ensuremath{#1}\xspace}
\newcommand{\Pg}{\xmath{P_{_{\mathnormal{GEV}}}}}

\newcommand{\Fa}{\xmath{\bar{F}_{\alpha,\delta}}}
\newcommand{\Fb}{\xmath{\bar{F}_{1,\delta}}}
\newcommand{\Pref}{\xmath{P_{_{\mathnormal{ref}}}}}
\newcommand{\Eref}{\xmath{E_{_{\Pref}}}}
\newcommand{\Ep}{\xmath{E_{_P}}}
\newcommand{\Eq}{\xmath{E_{_Q}}}
\newcommand{\bd}{\xmath{\bar{\delta}}}

\newcommand{\var}{\xmath{\textnormal{VaR}}}

\newcommand{\gref}{\xmath{\gamma_{_{\mathnormal{ref}}}}}

\begin{document}

\title{On Distributionally Robust Extreme Value Analysis}


\author{Jose Blanchet, Fei He, Karthyek Murthy}

\address{ Jose Blanchet: Management Science and Engineering, Stanford University, 475 via Ortega, Suite 310, Stanford, CA 94305.  Fei He: Department of IEOR, Columbia University, 500 West 120th Street, New York, NY 10027. Karthyek Murthy: Engineering Systems and Design, Singapore University of Technology and Design, 8 Somapah Road, Singapore, 487372. }

\email{jose.blanchet@stanford.edu, fh2293@columbia.edu, karthyek\_murthy@sutd.edu.sg} 


\date{}

\begin{abstract}
  We study distributional robustness in the context of Extreme Value
  Theory (EVT). We provide a data-driven method for
  estimating extreme quantiles in a manner that is robust against
  incorrect model assumptions underlying the application of the
  standard Extremal Types Theorem. Typical studies in distributional
  robustness involve computing worst case estimates over a model
  uncertainty region expressed in terms of the Kullback-Leibler
  discrepancy. We go beyond standard distributional robustness in that
  we investigate different forms of discrepancies, and prove rigorous
  results which are helpful for understanding the role of a putative
  model uncertainty region in the context of extreme quantile
  estimation. Finally, we illustrate our data-driven method in various
  settings, including examples showing how standard EVT can
  significantly underestimate quantiles of interest.
  
\end{abstract}

\maketitle

\section{Introduction}
Extreme Value Theory (EVT) provides reasonable statistical principles which
can be used to extrapolate tail distributions, and, consequently, estimate
extreme quantiles. However, as with any form for extrapolation, extreme
value analysis rests on assumptions that are rather difficult (or
impossible)\ to verify. Therefore, it makes sense to provide a mechanism to
robustify the inference obtained via EVT.

The goal of this paper is to study non-parametric distributional
robustness (i.e. finding the worst case distribution within some
discrepancy of a natural baseline model) in the context of EVT. We
ultimately provide a data-driven method for estimating extreme
quantiles in a manner that is robust against possibly incorrect model
assumptions. Our objective here is different from standard statistical
robustness which is concerned with data contamination only (not model
error); see, for example, \cite{Tsai2010}, for this type of analysis
in the setting of EVT. 

Our focus in this paper is closer in spirit to distributionally robust
optimization as in, for instance,
\cite{dupuis2000robust, Hans_Sarg, Ben_Tal, MAFI:MAFI12050}.  However, in
contrast to the 
literature on robust optimization, the emphasis here is on
understanding the implications of distributional uncertainty regions
in the context of EVT. As far as we know this is the first paper that
studies distributional robustness in the context of
EVT. 

We now describe the content of the paper, following the logic which
motivates the use of EVT.

\subsection{Motivation and Standard Approach}

In order to provide a more detailed description of the content of this
paper, its motivations, the specific contributions, and the methods
involved, let us invoke a couple of typical examples which motivate
the use of extreme value theory. As a first example, consider the
problem of forecasting the necessary strength that is required for a
skyscraper in New York City to withstand a wind speed that gets
exceeded only about once in 1000 years, using wind speed data that is
observed only over the last 200
years. 
In another instance, given the losses observed during the last few
decades, a reinsurance firm may want to compute, as required by
Solvency II standard, a capital requirement that is needed to
withstand all but about one loss in 200 years.

These tasks, and many others in practice, present a common challenge
of extrapolating tail distributions over regions involving unobserved
evidence from available observations. There are many reasonable ways
of doing these types of extrapolations. One might take advantage of
physical principles and additional information, if available, in the
windspeed setting; or use economic principles in the reinsurance
setting. In the absence of any fundamental principles which inform
tail extrapolation of a random variable $X,$ one may opt to use purely
statistical considerations. 

One such statistical approach entails the application of the popular
extremal types theorem (see Section \ref{SEC-EVT-STANDARD}) to model
the distribution of block maxima of a modestly large number of samples
of $X,$ by a generalized extreme value (GEV) distribution. Once we
have a satisfactory model for the distribution of
$M_n = \max\{X_1,\ldots,X_n\},$ evaluation of any desired quantile of
$X$ is straighforward because of the relationship that
$P(M_n \leq x) = (P(X \leq x))^n$ for any $x \in \mathbb{R}.$ Another
common approach is to use samples that exceed a certain threshold to
model conditional distribution of $X$ exceeding the threshold. The
standard texts in extreme value theory (see, for example,
\cite{MR691492,MR2234156,MR2364939}) provide a comprehensive account
of such standard statistical approaches.

Regardless of the technique used, various assumptions underlying an
application of a result similar to the extremal types theorem might be
subject to model error. Consequently, it has been widely accepted that
tail risk measures, particularly for high confidence levels, can only
be estimated with considerable statistical as well as model
uncertainty (see, for example, \cite{jorion2006value}).  The following
remark due to \cite{MR1932132} holds significance in this discussion:
\textquotedblleft Though the GEV model is supported by mathematical
argument, its use in extrapolation is based on unverifiable
assumptions, and measures of uncertainty on return levels should
properly be regarded as \textit{lower bounds} that could be much
greater if uncertainty due to model correctness were taken into
account.\textquotedblright

Despite these difficulties, however, EVT is widely used (see, for example,
\cite{MR2234156}) and regarded as a reasonable way of extrapolation to
estimate extreme quantiles.

\subsection{Proposed Approach Based on Infinite Dimensional Optimization}
We share the point of view that EVT is a reasonable approach, so we
propose a procedure that builds on the use of EVT to provide upper
bounds which attempts to address the types of errors discussed in the
remark above from \cite{MR1932132}.  For large values of $n$, under
the assumptions of EVT, the distribution of $M_{n}$ lies close to, and
appears like, a GEV distribution. Therefore, instead of considering
only the GEV distribution as a candidate model, we propose a
non-parametric approach. In particular, we consider a family of
probability models, all of which lie in a \textquotedblleft
neighborhood\textquotedblright\ of a GEV model, and compute a
conservative worst-case estimate of Value atrisk (VaR) over all of
these candidate models. For $p \in [0,1],$ the value at risk
$\var_p(X)$ is defined as 
\[ \var_p(X) = F^{\leftarrow}(p) := \inf\{ x: P\{ X \leq x\} \geq p
\}.\]

Mathematically, given a reference model, $\Pref$, which we consider to
be obtained using EVT (using a procedure such as the one outlined in the
previous subsection), we consider the optimization problem
\begin{equation}
\sup \bigg\{P \{X>x\}:\ d(P, \Pref\xspace)\leq \delta %
\bigg\}.
\label{P-FAM}
\end{equation}%
Note that the previous problem proposes optimizing over all probability
measures that are within a tolerance level $\delta $ (in terms of a suitable
discrepancy measure $d$) from the chosen baseline reference model $\Pref.$

There is a wealth of literature that pursues this line of thought (see
\cite{dupuis2000robust,
  Hans_Sarg, Ahmadi-Javid2012, Ben_Tal,MAFI:MAFI12050,Glasserman_Xu}%
), but, no study has been carried out in the context of EVT. Moreover,
while the solvability of problems as in \eqref{P-FAM} have
understandably received a great deal of attention, the qualitative
differences that arise by using various choices of discrepancy
measures, $d$, has not been explored, and this is an important
contribution of this paper. For tractability reasons, the usual choice
for discrepancy $d$ in the literature has been
KL-divergence. 
In Section \ref{SEC-NONPARAM-ROB} we study the solution to infinite
dimensional optimization problems such as \eqref{P-FAM} for a large
class of discrepancies that includes KL-divergence as a special 
case, and discuss how such problems can be solved at no significant
computational cost.

\subsection{Choosing Discrepancy and Consistency Results}

One of our main contributions in this paper is to systematically demonstrate
the qualitative differences that arise by using different choices of
discrepancy measures $d$ in \eqref{P-FAM}. Since our interest in the paper
is limited to robust tail modeling via EVT, this narrow scope, in turn, let
us analyse the qualitative differences that may arise because of different
choices of $d$.

As mentioned earlier, the KL-divergence\footnote{%
  KL-divergence, and all other relevant divergence measures, are
  defined in Section \ref{SEC-DIV-MEAS}} is the most popular choice
for $d$. In Section %
\ref{SEC-ASYMP-CHAR} we show that for any divergence neighborhood
$\mathcal{P%
}$,
defined using $d=$ KL-divergence around a baseline reference $\Pref$,
there exists a probability measure $P$ in $\mathcal{P}$ that has tails
as heavy as
\begin{equation*}
P(x,\infty )\geq c\log ^{-2}\Pref\xspace(x,\infty ),
\end{equation*}%
for a suitable constant $c$, and all large enough $x.$ This means,
irrespective of how small $\delta $ is (smaller $\delta $ corresponds to
smaller neighborhood $\mathcal{P}$), a KL-divergence neighborhood around a
commonly used distribution (such as exponential, (or) Weibull (or) Pareto)
typically contains tail distributions that have infinite mean or variance,
and whose tail probabilities decay at an unrealistically slow rate (even
logarithmically slow, like $\log ^{-2}x$, in the case of reference models
that behave like a power-law or Pareto distribution). As a result,
computations such as worst-case expected short-fall\footnote{%
Similar to VaR, expected shortfall (or) conditional value at risk (referred
as CVaR) is another widely recognized risk measure.} may turn out to be
infinite. Such worst-case analyses are neither useful nor interesting.%

For our purposes, we also consider Renyi divergence
measures $D_{\alpha }$ (see Section \ref{SEC-DIV-MEAS}) that includes KL-divergence as a special case (when
$\alpha =1$%
). It turns out that for any $\alpha >1$, the divergence neighborhoods
defined as in
$\{P:D_{\alpha }(P,P_{_{\mathnormal{ref}}}\xspace)\leq \delta \}$
consists of tails that are heavier than
$P_{_{\mathnormal{ref}}}\xspace$%
, but not prohibitively heavy. More importantly, we prove a
\textquotedblleft consistency\textquotedblright\ result in the sense
that if the baseline reference model belongs to the maximum domain of
attraction of a GEV distribution with shape parameter
$\gamma _{_{\mathnormal{ref}}}\xspace,$ then the corresponding
worst-case tail distribution,
\begin{equation}
\Fa(x):=\sup \{P(x,\infty ):D_{\alpha }(P,P_{_{%
\mathnormal{ref}}}\xspace)\leq \delta \},  \label{EQ_WHOLE_TAIL}
\end{equation}%
belongs to the maximum domain of attraction of a GEV distribution with
shape parameter
$\gamma ^{\ast }=(1-\alpha ^{-1})^{-1}\gamma _{_{\mathnormal{ref}}}%
\xspace$ (if it exists).

Since our robustification approach is built resting on EVT principles,
we see this consistency result as desirable. If a modeler who is
familiar with certain type of data expects the EVT inference to result
in an estimated shape parameter which is positive, then the
robustification procedure should preserve this qualitative
property. An analysis of the maximum domain of attraction of the
distribution $\bar{F}_{\alpha }\xspace(x)$, depending on $\alpha $ and
$\gref,$ is presented in Section \ref{SEC-ASYMP-CHAR}, along with a
summary of the results in Table 1.

Note that the smaller the value of $\alpha $, the larger the absolute
value of shape parameter $\gamma^{\ast }$, and consecutively, heavier
the corresponding worst-case tail is. This indicates a gradation in
the rate of decay of worst-case tail probabilities as parameter
$\alpha $ decreases to 1, with the case $\alpha =1$ (corresponding to
KL-divergence) representing the extreme heavy-tailed behaviour. This
gradation, as we shall see, offers a great deal of flexibility in
modeling by letting us incorporate domain knowledge (or) expert
opinions on the tail behaviour. If a modeler is suspicious about the
EVT inference he/she could opt to select $\alpha =1$, but, as we have
mentioned earlier, this selection may result in pessimistic estimates.

The relevance of these results shall become more evident as we
introduce the required terminology in the forthcoming
sections. Meanwhile, Table \ref%
{TAB-DA-DIFF-CASES} 
offers illustrative comparisons of $\bar{F}_{\alpha }\xspace(x)$ for various
choices of $\alpha.$

\subsection{The Final Estimation Procedure}
The framework outlined in the previous subsections yields a data
driven procedure for estimating VaR which is presented in Section
\ref{SEC-ROB-VAR-EST}. A summary of the overall procedure is given in
Algorithm 2. The procedure is applied to various data sets, resulting
in different reference models, and we emphasize the choice of
different discrepancy measures via the parameter $\alpha $. The
numerical studies expose the salient points discussed in the previous
subsections and rigorously studied via our theorems. For instance,
Example 3 shows how the use of the KL divergence might lead to rather
pessimistic estimates. Moreover, Example 4 illustrates how the direct
application of EVT can severely underestimate the quantile of
interest, while the procedure that we advocate provides correct
coverage for the extreme quantile of interest.

The very last section of the paper, Section \ref{SEC-THM-PROOF}, contains
technical proofs of various results invoked in the development.

\section{Generalized extreme value distributions}
\label{SEC-EVT-STANDARD}
The objective of this section is to mainly fix notation and review
properties of generalized extreme value (GEV) distributions that are
relevant for introducing and proving our main results in Section
\ref{SEC-ASYMP-CHAR}. For a thorough introduction to GEV distributions
and their applications to modeling extreme quantiles, we refer the
readers to the wealth of literature that is available (see, for
example, \cite{MR691492, EKM97, MR2234156, MR2364939} and references
therein). 

If we use $M_n$ to denote the maxima of $n$ independent copies of a
random variable $X$ with cumulative distribution funtion $F(\cdot),$
then extremal types theorem identifies all non-degenerate
distributions $G(\cdot)$ that may occur in the limiting relationship,
\begin{align}
  \label{MAX_LIM_DIST}
  \lim_{n \rightarrow \infty} P \left\{ \frac{M_n - b_n}{a_n} \leq x
  \right\} = \lim_{n \rightarrow \infty} F^n\left( a_n x + b_n \right)
  = G(x),  
\end{align}
for every continuity point $x$ of $G(\cdot),$ with $a_n$ and $b_n$
representing suitable scaling constants.  All such distributions
$G(x)$ that occur in the right-hand side of \eqref{MAX_LIM_DIST}
are called \textit{extreme value distributions}.\\
 
\noindent
\textbf{Extremal types theorem} (Fisher and Tippet (1928),
Gnedenko (1943)). The class of extreme value distributions is $G_\gamma
(ax + b)$ with $a > 0, b, \gamma \in \mathbb{R},$ and 
\begin{align}
\label{EV-DISTS}
 G_\gamma(x) := \exp \left( -\left( 1 + \gamma x
  \right)^{-1/\gamma}\right), \quad\quad 1 + \gamma x > 0.
\end{align}
If $\gamma = 0,$ the right-hand side is interpreted as
$\exp(-\exp(-x)).$\\

\noindent The extremal types theorem asserts that any $G(x)$ that
occurs in the right-hand side of \eqref{MAX_LIM_DIST} must be of the
form $G_\gamma(ax+b).$ As a convention, any probability distribution
$F(x)$ that gives rise to the limiting distribution
$G(x) = G_\gamma(ax + b)$ in \eqref{MAX_LIM_DIST} is said to belong to
the maximum domain of attraction of $G_\gamma(x).$ In short, it is
written as $F \in \mathcal{D}(G_\gamma).$ The parameters
$\gamma, a>0 $ and $b$ are, respectively, called the shape, scale and
location parameters. From the above we have
\[
  P(M_n\leq x) = P\Big(\frac{M_n - b_n}{a_n}\leq \frac{x - b_n}{a_n}\Big) \approx G_{\gamma_0}\Big(\frac{x - b_n}{a_n}\Big) =: G_{\gamma_0}(a_0 x+b_0),
\]
where $\gamma_0, a_n, b_n$ are estimated by a parameter estimation technique such as maximum likelihood and $a_0 := 1/a_n$, $b_0 := -b_n/a_n$. We will use $P_{GEV}$ to denote the distribution $G_{\gamma_0}(a_0 x+b_0)$.

\subsection{Frechet, Gumbel and Weibull types}
Though the limiting distributions $G_\gamma(ax+b)$ seem to constitute
a simple parametric family, they include a wide-range of tail
behaviours in their maximum domains of attraction, as discussed below:
For a distribution $F,$ let $\bar{F}(x) = 1-F(x)$ denote the
corresponding tail probabilities, and $x_{_F}^* = \sup\{x: F(x) < 1\}$
denote the right endpoint of its support.
\begin{itemize}
\item[1)] \textbf{The Frechet Case $(\gamma > 0).$}  A distribution
  $F \in \mathcal{D}(G_\gamma)$ for some $\gamma > 0,$ if and only if 
  right endpoint $x_{_F}^*$ is unbounded, and its tail probabilities
  satisfy 
  \begin{align}
    \label{FRECH-CHAR}
    \bar{F}(x) = \frac{L(x)}{x^{1/\gamma}}, \quad\quad x > 0
  \end{align}
  for a function $L(\cdot)$ slowly varying at $\infty$\footnote{A
    function $L: \mathbb{R} \rightarrow \mathbb{R}$ is said to be
    slowly varying at infinity if
    $\lim_{x \rightarrow \infty} L(tx)/L(x) = 1$ for every $t > 0.$
    Common examples of slowly varying function include
    $\log x, \log \log x,1-\exp(-x),$ constants, etc.}.  As a
  consequence, moments greater than or equal to $1/\gamma$ do not
  exist. Any distribution $F(x)$ that lies in $\mathcal{D}(G_\gamma)$
  for some $\gamma > 0$ is also said to belong to the maximum domain
  of attraction of a Frechet distribution with parameter $1/\gamma.$
  The Pareto distribution $1-F(x) = x^{-\alpha} \wedge 1$ is an
  example for a distribution that belongs to
  $\mathcal{D}(G_{1/\alpha}).$ 

\item[2)] \textbf{The Weibull case $(\gamma < 0).$} Unlike the Frechet
  case, a distribution $F \in \mathcal{D}(G_\gamma)$ for some
  $\gamma < 0,$ if and only if its right endpoint $x_{_F}^*$ is finite,
  and its tail probabilities satisfy
    \begin{align}
      \label{WEIB-CHAR}
      \bar{F}(x_{_F}^*- \epsilon) =
      \epsilon^{-1/\gamma}L\left(\frac{1}{\epsilon}\right), 
      \quad\quad \epsilon > 0
    \end{align}
    for 
    a function $L(\cdot)$ slowly varying at $\infty.$ A distribution
    that belongs to $\mathcal{D}(G_\gamma)$ for some $\gamma < 0$ is
    also said to belong to the maximum domain of attraction of Weibull
    family.  The uniform distribution on the interval $[0,1]$ is an
    example that belongs to this class of extreme value distributions.

  \item[3)] \textbf{The Gumbel case $(\gamma = 0).$} A distribution $F
    \in \mathcal{D}(G_0)$ if and only if 
  \begin{align}
    \lim_{t \uparrow x_{_F}^*} \frac{\bar{F}(t+xf(t))}{\bar{F}(t)} =
    \exp(-x), \quad\quad x \in \mathbb{R}
    \label{GUMB-CHAR}
  \end{align}
  for a suitable positive function $f(\cdot).$ In general, the members
  of $G_0$ have exponentially decaying tails, and consequently, all
  moments exist. Probability distributions $F(\cdot)$ that give rise
  to limiting distributions $G_0(ax + b)$ are also said to belong to
  the Gumbel domain of attraction. Common examples that belong to the
  Gumbel domain of attraction include exponential and normal
  distributions.
\end{itemize}
Given a distribution function $F,$ Proposition
\ref{PROP-EXTREMES-CHAR} is useful to test to determine its domain
of attraction:
\begin{proposition}
  \textnormal{Suppose $F''(x)$ exists and $F'(x)$ is positive for all
    $x$ in some left neighborhood of $x_{_F}^*.$ If
  \begin{align}
    \lim_{x \uparrow x_{_F}^*} \left( \frac{1-F}{F'} \right)'\hspace{-5pt}(x) =
    \gamma,
\end{align}
then $F$ belongs to the domain of attraction of $G_\gamma.$}
 \label{PROP-EXTREMES-CHAR}
\end{proposition}
The proof of Proposition \ref{PROP-EXTREMES-CHAR} and further details
on the classification of extreme value distributions can be found in
any standard text on extreme value theory (see, for example,
\cite{MR691492} or \cite{MR2234156}).

\subsection{On model errors and robustness}
After identifying a suitable GEV model $\Pg$ for the distribution of
block maxima $M_n$, it is common to utilize the relationship
$P\{ M_n \leq x\} = P\{ X \leq x\}^n,$ to compute a desired extreme
quantile of $X.$ It is useful to remember that $\Pg(-\infty,x]$ is
only an approximation for $P\{ M_n \leq x\},$ and the quality of the
approximation is, in turn, dependent on the unknown distribution
function $F$ (see \cite{MR2364939,MR2234156}). Therefore, in practice,
one does not know the block-size $n$ for which the GEV model $\Pg$
well-approximates the distribution of $M_n.$ Even if a good choice of
$n$ is known, one cannot often employ it in practice, because larger
$n$ means smaller $m$, the number of blocks, and consequentially, the inferential errors
could be large. 
Due to the arbitrariness in the
estimation procedures and the nature of applications (calculating wind
speeds for building sky-scrapers, building dykes for preventing
floods, etc.), it is desirable to have, in addition, a data-driven
procedure that yields a conservative upper bound for $x_p$ that is
robust against model errors. To accomplish this, one can form a
collection of competing probability models $\mathcal{P},$ all of which
appear plausible as the distribution of $M_n$, and compute the maximum
of $p^n$-th quantile over all the plausible models in $\mathcal{P}.$
This is indeed the objective of the sections that follow.


\section{A non-parametric framework for addressing model errors} 
\label{SEC-NONPARAM-ROB}
Let $(\Omega,\mathcal{F})$ be a measurable space and
$M_1(\mathcal{F})$ denote the set of probability measures on
$(\Omega, \mathcal{F}).$ Let us assume that a reference probability
model $\Pref \in M_1(\mathcal{F})$ is inferred by suitable modelling
and estimation procedures from historical data. Naturally, this model
is not the same as the distribution from which the data has been
generated, and is expected only to be close to the data generating
distribution. In the context of Section \ref{SEC-EVT-STANDARD}, the
model $\Pref$ corresponds to $\Pg,$ and the data generating model
corresponds to the true distribution of $M_n.$ With slight
perturbations in data, we would, in turn, be working with a slightly
different reference model. Therefore, it has been of recent interest
to consider a family of probability models $\mathcal{P},$ all of which
are plausible, and perform computations over all the models in that
family.  Following the rich literature of robust optimization, where
it is common to describe the set of plausible models using distance
measures (see \cite{Ben_Tal}), we consider the set of plausible models
to be of the form
\[ \mathcal{P} = \left\{ P \in M_1(\mathcal{F}) : d\left( P,
      \Pref\right) \leq \delta \right\}\] for some distance functional
$d:M_1(\mathcal{F}) \times M_1(\mathcal{F}) \rightarrow \mathbb{R}_+
\cup \{+\infty\},$ and a suitable $\delta > 0.$ Since
$d(\Pref,\Pref) = 0$ for any reasonable distance functional, $\Pref$
lies in $\mathcal{P}.$ Therefore, for any random variable $X,$ along
with the conventional computation of $\Eref [X],$ one aims to provide
``robust'' bounds,
\[ \inf_{P \in \mathcal{P}} \Ep [X] \leq \Eref [X] \leq \sup_{P \in
  \mathcal{P}} \Ep [X].\]
Here, we follow the notation that $\Ep [X] = \int X dP$ for any
$P \in M_1(\mathcal{F}).$ Since the state-space $\Omega$ is
uncountable, evaluation of the above $\sup$ and $\inf$-bounds, in
general, are infinite-dimensional problems. However, as it has been
shown in the recent works \cite{MAFI:MAFI12050,Glasserman_Xu}, it is
indeed possible to evaluate these robust bounds for carefully chosen
distance functionals $d.$

\subsection{Divergence measures}
\label{SEC-DIV-MEAS}
Consider two probability measures $P$ and $Q$ on
$(\Omega, \mathcal{F})$ such that $P$ is absolutely continuous with
respect to $Q.$ The Radon-Nikodym derivative $dP/dQ$ is then
well-defined. The Kullback-Liebler divergence (or KL-divergence) of
$P$ from $Q$ is defined as
\begin{equation} 
  D_1(P,Q) := \Eq \left[ \frac{dP}{dQ} \log \left(
      \frac{dP}{dQ}\right)\right].
\label{KL-DIV}
\end{equation}
This quantity, also referred to as relative entropy (or) information
divergence, arises in various contexts in probability
theory
. For our purposes, it will be useful to consider a general class of
divergence measures that includes KL-divergence as a special case. For
any $\alpha > 1,$ the R\'{e}nyi divergence of degree $\alpha$ is
defined as:
\begin{equation}
  D_\alpha(P,Q) := \frac{1}{ \alpha - 1} \log \Eq \left[
    \left(\frac{dP}{dQ}\right)^\alpha \right].
\label{REN-DIV}
\end{equation}
It is easy to verify that for every $\alpha,$ $D_\alpha(P,Q) = 0,$ if
and only if $P = Q.$ Additionally, the map $\alpha \mapsto D_\alpha$
is nondecreasing, and continuous from the left. Letting
$\alpha \rightarrow 1$ in \eqref{REN-DIV} yields the formula for
KL-divergence $D_1(P,Q).$ Thus KL-divergence is a special case of the
family of R\'{e}nyi divergences, when the parameter $\alpha$ equals
$1.$ If the probability measure $P$ is not absolutely continuous with
respect to $Q,$ then $D_\alpha(P,Q)$ is taken as $\infty.$ Though none
of these divergence measures form a metric on the space of probability
measures, they have been used in a variety of scientific disciplines
to discriminate between probability measures. For more details on the
divergences $D_\alpha,$ see \cite{MR0132570,MR926905}.

\subsection{Robust bounds via maximization of convex integral
  functionals}
\label{SEC-ROBUST-GEN-THEORY}
Recall that $\Pref$ is the reference probability measure obtained via
standard estimation procedures. Since the model $\Pref$ could be
misspecified, we consider all models that are not far from $\Pref$ in
the sense quantified by divergence $D_\alpha,$ for any fixed
$\alpha \geq 1.$ Given a random variable $X,$ we consider optimization
problems of form
\begin{align}
  V_\alpha(\delta) := \sup \left\{ \Ep [X]: D_\alpha(P,\Pref) \leq
  \delta \right\}. 
\label{OPT-GEN-PROB}
\end{align} 
Though KL-divergence has been a popular choice in defining sets of
plausible probability measures as above, use of divergences $D_\alpha,
\ \alpha \neq 1$ is not new altogether: see \cite{MR3299141,Glasserman_Xu}. Due to the Radon-Nikodym theorem, $V_\alpha(\delta)$ can
be alternatively written as,
\begin{align}
  V_\alpha(\delta) = \sup \left\{ \Eref [LX]:
  \Eref[\phi_\alpha(L)] \leq \bd, \Eref [L] = 1, L \geq 0\right\},
\label{OPT-PRIMAL}
\end{align}
where $L= dP/dP_{ref}$ and 
\begin{align}
  \phi_\alpha(x) =
  \begin{cases}
    x^\alpha \quad\quad\ \  \text{ if } \alpha > 1,\\
    x\log x \quad\text{ if } \alpha = 1
  \end{cases}
  \text{ and }
  \quad \bd = 
  \begin{cases}
    \exp \left((\alpha-1)\delta \right) \quad \ \text{ if } \alpha > 1,\\
    \delta \quad\quad \quad\quad \quad\quad \quad \text{ if } \alpha =
    1.
  \end{cases}
\label{DEFN-PHI-BD}
\end{align}

A standard approach for solving optimization problems of
the above form is to write the corresponding dual problem as below:
\begin{align*}
  V_\alpha(\delta) &\leq \inf_{\overset{\lambda \geq 0,}{\mu}} \sup_{L
                     \geq 0} \Eref \left[ L X - \lambda \left(
                     \phi_\alpha(L) - \bd \right) + \mu (L-1)\right]. 
\end{align*}
The above dual problem can, in turn, be relaxed by taking the $\sup$
inside the expectation:
\begin{equation}
  V_\alpha(\delta) \leq \inf_{\overset{\lambda \geq 0,}{\mu}} \left\{
                     \lambda \bd - \mu + \lambda \Eref
                     \left[\sup_{L \geq 0} \left\{ \frac{(X +
                     \mu)}{\lambda} L -\phi_\alpha(L)
                     \right\}\right]\right\}. \label{DUAL-PROB-BNDD}                       
\end{equation}
By first order condition the inner supremum is solved by
\begin{align}
  L^*_\alpha(c_1,c_2) :=
  \begin{cases}
    c_1 \exp(c_2 X), &\text{ if } \alpha = 1,\\
    (c_1 + c_2X)_+^{1/(\alpha-1)}, &\text{ if } \alpha > 1,
  \end{cases}
\label{GEN-OPTIMIZER}
\end{align}
for some suitable constants $c_1\in\mathbb{R}, c_2>0$ when $\alpha>1$;
and $c_1\in (0,1)$ and $c_2>0$ when $\alpha=1$. Then the following
result is intuitive:
\begin{proposition}
  Fix any $\alpha \geq 1.$ For $L^*_\alpha(c_1,c_2)$ defined as in
  \eqref{GEN-OPTIMIZER}, if there exists constants $c_1$ and $c_2$
  such that
  \[ L^*_\alpha(c_1,c_2) \geq 0, \ \Eref \left[ L^*_\alpha(c_1,c_2)
  \right] = 1 \text{ and } \Eref\left[ \phi_\alpha
    \left(L^*_\alpha(c_1,c_2) \right) \right] = \bd,\]
  then $L^*_\alpha(c_1,c_2)$ solves the optimization problem
  \eqref{OPT-PRIMAL}. The corresponding optimal value is
\begin{align}
 V_\alpha(\delta) = \Eref \left[ L^*_\alpha(c_1,c_2) X \right].
\label{GEN-OPT-VAL}
\end{align}
\label{THM-GEN-OPT}
\end{proposition}
\begin{proof}
  Under the specified assumptions, when we plug $L^*_\alpha(c_1,c_2)$
  into the right-hand-side of inequality (\ref{DUAL-PROB-BNDD}), it is
  simplified to $\Eref \left[ L^*_\alpha(c_1,c_2) X \right]$, so we
  have
  $V_{\alpha}(\delta)\leq \Eref \left[ L^*_\alpha(c_1,c_2) X
  \right]$. On the other hand, since $L^*_\alpha(c_1,c_2)$ satisfies
  all the constraints in the problem (\ref{OPT-PRIMAL}), we have
  $V_{\alpha}(\delta)\geq \Eref \left[ L^*_\alpha(c_1,c_2) X
  \right]$. \hfill $\Box$
\end{proof}

\begin{remark}
  \textnormal{Let us say one can determine constants $c_1$ and $c_2$
    for given $X, \alpha$ and $\delta.$ Then, as a consequence of
    Proposition \ref{THM-GEN-OPT}, the optimization problem
    \eqref{OPT-GEN-PROB} involving uncountably many measures can, in
    turn, be solved by simply simulating $X$ from the original
    reference measure $\Pref,$ and multiplying by corresponding
    $L^*_\alpha(c_1,c_2)$ to compute the expectation as in
    \eqref{GEN-OPT-VAL}. Interested readers are referred to
    \cite{Glasserman_Xu} for specific examples illustrating this
    procedure. A general theory for optimizing convex integral
    functionals of form \eqref{OPT-PRIMAL}, that includes a bigger
    class of divergence measures, can be found in
    \cite{MAFI:MAFI12050}.}
\end{remark}


\noindent In this paper, we restrict to the case where the random
variable $X$ above is an indicator function. As illustrated in Section
\ref{sec:wc-prob} below, the computation of bounds $V_\alpha(\delta)$
turns out to be simpler for this special case. 

\subsection{Evaluation of worst case probabilities}
\label{sec:wc-prob}
From here onwards, suppose that $\Pref$ is a
probability measure on $(\mathbb{R}, \mathcal{B}(\mathbb{R}))$
satisfying $\Pref(x,\infty) \rightarrow 0$ as
$x \rightarrow \sup\{x:\Pref(x,\infty) > 0\}.$ For a given
$\delta > 0,$ $\alpha \geq 1,$ define the worst-case tail probability
function, $\Fa:\mathbb{R} \rightarrow [0,1],$ as,
\begin{align}
  \Fa(x) := \sup \{ P(x,\infty): D_\alpha(P,\Pref) \leq \delta\}.
  \label{defn:worst-case-tail}
\end{align}
In addition, for a given $\alpha \geq 1,$ define the functions
\begin{equation*}
    \delta_{thr}(u) := u\phi_\alpha(1/u) \quad \text{ for } \quad u \in (0,1) 
\end{equation*}
and
\begin{equation*}
  g_\alpha(u,\theta) :=     u\phi_\alpha(\theta) + (1-u)\phi_\alpha\left( \frac{1-\theta
  u}{1-u}\right) \quad\text{for}\quad \{(u,\theta) \in (0,1) \times (1,\infty): u\theta \leq 1\}.
\end{equation*}   

The following result is a corollary of Proposition \ref{THM-GEN-OPT}.
\begin{corollary}
  \label{COR-TP-SOLN-CHAR}
  Suppose that $\Fa(\cdot)$ is defined as in
  \eqref{defn:worst-case-tail} and $x \in \mathbb{R}$ is such that
  $\Pref(x,\infty) > 0.$ Then, if
  $\bar{\delta} \leq \delta_{thr}(\Pref(x,\infty)),$ there exists
  $\theta_x > 1$ satisfying,
  \begin{align}
    g_\alpha(\Pref(x,\infty),\theta_x) = \bar{\delta}.
    \label{THETA-SOLN-EQN}
  \end{align}  
  Moreover,
  \begin{align}
      \Fa(x) =  
      \begin{cases}
        \theta_x P_{ref}(x,\infty) \quad\quad&\text{ if } \bar{\delta} \leq \delta_{thr}(\Pref(x,\infty)),\\
        1 & \text{ otherwise.}
      \end{cases}
    \end{align}
\end{corollary}
\begin{proof}
  Consider the canonical mapping
  $Z(\omega) = \omega, \ \omega \in \mathbb{R}.$ Then, for a given
  $x,$
    \begin{align*}
      \Fa(x) = \sup \left\{ E_{\Pref}[ L\mathbf{1}(Z > x)] :
      E_{_{\Pref}}\left[ \phi_\alpha(L) \right] \leq 
      \bd, E_{_{\Pref}} [L] = 1, L \geq 0 \right\}.
    \end{align*}
    is an optimization problem of the
    form \eqref{OPT-GEN-PROB}. Therefore, due to Proposition
    \ref{THM-GEN-OPT} and equation \eqref{GEN-OPTIMIZER}, the optimal $L^*$ has the form
    \begin{align*}
      L^*_\alpha(c_1,c_2) :=
      \begin{cases}
        c_1 \exp(c_2 \mathbf{1}(Z>x)), &\text{ if } \alpha = 1,\\
        (c_1 + c_2\mathbf{1}(Z>x))_+^{1/(\alpha-1)}, &\text{ if } \alpha > 1,
      \end{cases}
    \end{align*}
    When we consider the two cases of $Z>x$ and $Z\leq x$, and combine
    the range information on $c_1, c_2$ following equation
    (\ref{GEN-OPTIMIZER}), the above formulation of
    $L^*_\alpha(c_1,c_2)$ can further be simplified to
    $\theta_x \mathbf{1}(x,\infty) + \tilde{\theta}_x
    \mathbf{1}(-\infty,x]$ for some constants $\theta_x > 1$ and
    $\tilde{\theta}_x \in [0,1).$ Substituting
    \[L^*_\alpha = \theta_x \mathbf{1}(x,\infty) + \tilde{\theta}_x
      \mathbf{1}(-\infty,x]\] in the constraints $L^\ast \geq 0,$
    $E_{_{\Pref}}[\phi_\alpha(L^*)] = \bd$ and
    $E_{_{\Pref}}[L^*] = 1,$ we obtain that for any $\theta_x$ and
    $\tilde{\theta}_x$ satisfying, 
    \begin{align*}
      \theta_x &\in \left(1,1/\Pref(x,\infty)\right], \quad 
      \tilde{\theta}_x = \frac{1-\theta_x \Pref(x,\infty)}{1- \Pref(x,\infty)} \in [0,1) \quad \text{ and }\\  \bd &= \Pref(x,\infty) \phi_\alpha(\theta_x) + (1-\Pref(x,\infty))\phi_\alpha(\tilde{\theta}_x) = g_\alpha(\Pref(x,\infty),\theta_x), 
    \end{align*}
    we have, 
    \[\Fa(x) = E_{_{\Pref}}\left[L^\ast_\alpha \mathbf{1}(Z > x)\right] =
      \theta_x P_{ref}(x,\infty).\]

    Next, for any fixed $u \in (0,1),$ observe that
    $g_\alpha(u,\theta)$ is increasing continuously in $\theta$ over
    the interval $(1,1/u],$ taking values in the range
    $(1,\delta_{thr}(u)]$ when $\alpha>1$, and in the range $(0,\delta_{thr}(u)]$ when $\alpha=1$. Therefore, an assignment of $\theta$
    satisfying $g_\alpha(u,\theta) = \bd$ exists only when
    $\bd \leq \delta_{thr}(u).$ 
    In particular, the assignment $\theta$ satisfying
    $g_\alpha(u,\theta) = \bd$ increases as $\bd$ increases until
    when $\bd = \delta_{thr}(u)$ for which the corresponding $\theta$
    satisfying $g_\alpha(u,\theta) = \delta_{thr}(u)$ is given by
    $\theta = 1/u.$

    Thus, given $x \in \mathbb{R}$ such that
    $\Pref(x,\infty) \in (0, 1),$ there exists $\theta_x > 1$
    satisfying \eqref{THETA-SOLN-EQN} only if
    \[\bd \leq \delta_{thr}(\Pref(x,\infty)),\] and specifically for the case, 
    $\bd = \delta_{thr}(\Pref(x,\infty)),$ we have
    $\theta_x = 1/\Pref(x,\infty).$ Therefore,
    \[\Fa(x) =
      \theta_x P_{ref}(x,\infty)
      =
      \begin{cases}
        \theta_x P_{ref}(x,\infty) \quad\quad&\text{ if } \bd < \delta_{thr}(\Pref(x,\infty)),\\
        1 &\text{ if } \bd = \delta_{thr}(\Pref(x,\infty))
      \end{cases}
    \]
    Since $\Fa(x)$ is nondecreasing in $\delta,$ it follows that
    $\Fa(x) = 1,$ also for values of $\delta$ such that the
    corresponding $\bd > \delta_{thr}(\Pref(x,\infty)).$ \hfill $\Box$
  \end{proof}

\section{Asymptotic analysis of robust estimates of tail
  probabilities } 
\label{SEC-ASYMP-CHAR}
In this section we study the asymptotic behaviour of 
$\Fa (x) := \sup \{ P(x,\infty) : D_\alpha(P, \Pref) \leq \delta\},$ 
for any $\alpha \geq 1$ and $\delta > 0,$ as $x \rightarrow \infty.$
We first verify in Proposition \ref{Prop-Fa-tail-cdf} below that
$\Fa(x),$ viewed as a function of $x,$ satisfies the properties of a
tail distribution function. A proof of Proposition
\ref{Prop-Fa-tail-cdf} is presented in Section \ref{SEC-THM-PROOF}.

\begin{proposition}
  The function, $F_{\alpha,\delta}(x):= 1-\Fa(x),$ viewed as a
  function of $x,$ satisfies properties of cumulative distribution
  function of a real-valued random variable.
\label{Prop-Fa-tail-cdf}
\end{proposition}

Thus from here onwards, we shall refer to $\Fa(\cdot)$ as the
\textit{$\alpha$-family worst-case tail distribution}, and study its
qualitative properties such as domain of attraction for the rest of
this section. All the probability measures involved, unless
explicitly specified, are taken to be defined on
$(\mathbb{R}, \mathcal{B}(\mathbb{R})).$ Since
$D_\alpha(\Pref,\Pref)=0,$ it is evident that the worst-case tail
estimate $\Fa(x)$ is at least as large as $\Pref(x,\infty).$ While the
overall objective has been to provide robust estimates that account
for model perturbations, it is certainly not desirable that the
worst-case tail distribution $\Fa(\cdot),$ for example, has
unrealistically slow logarithmic decaying tails. Seeing this, our
interest in this section is to quantify how heavier the tails of
$\Fa(\cdot)$ are, when compared to that of the reference model.

The bigger the plausible family of measures
$\left\{ P: D_\alpha(P,\Pref) \leq \delta \right\},$ the slower the
decay of tail $\Fa(x)$ is, and vice versa. Hence it is conceivable
that the parameter $\delta$ is influential in determining the rate of
decay of $\Fa(\cdot).$ However, as we shall see below in Theorem
\ref{EVT-RENYI}, it is the parameter $\alpha$ (along with the tail
properties of the reference model $\Pref$) that solely determines the
domain of attraction, and hence the rate of decay, of $\Fa(\cdot).$


Since our primary interest in the paper is with respect to reference
model $\Pref$ being a GEV model, we first state the result in this
context: 
\begin{theorem}
  Let the reference GEV model $\Pg$ have shape parameter $\gref.$ Then the distribution $F$ induced by $\Pg$ satisfies the regularity assumptions of Proposition \ref{PROP-EXTREMES-CHAR} with $\gamma = \gref.$ For
  any $\alpha > 1,$ let
  $\Fa (x) := \sup \{ P(x,\infty) : D_\alpha(P, \Pg) \leq \delta\},$
  and
  \[\gamma^\ast := \frac{\alpha}{\alpha-1}\gref.\]
  Then the distribution function $F_{\alpha,\delta}(x) = 1 - \Fa(x)$
  belongs to the domain of attraction of $G_{\gamma^\ast}.$ 
  \label{EVT-GEV-RENYI}
\end{theorem}

\noindent Theorem \ref{EVT-GEV-RENYI} is, however, a corollary of
Theorem \ref{EVT-RENYI} below.

\begin{theorem}
  Let the reference model $\Pref$ belong to the domain of attraction
  of $G_{\gamma_{\textnormal{ref}}}.$ In addition, let $\Pref$ induce
  a distribution $F$ that satisfies the regularity assumptions of
  Proposition \ref{PROP-EXTREMES-CHAR} with $\gamma = \gref.$ For
  any $\alpha > 1,$ let
  $\Fa (x) := \sup \{ P(x,\infty) : D_\alpha(P, \Pref) \leq \delta\},$
  and
  \[\gamma^\ast := \frac{\alpha}{\alpha-1}\gref.\]
  Then the distribution function $F_{\alpha,\delta}(x) = 1 - \Fa(x)$
  belongs to the maximum domain of attraction of $G_{\gamma^\ast}.$
  \label{EVT-RENYI}
\end{theorem}
The special case corresponding to $\alpha = 1$ is handled in
Propositions \ref{propspecial} and
\ref{PROP-KL}.  
Proofs of Theorems \ref{EVT-GEV-RENYI} and \ref{EVT-RENYI} are
presented in Section \ref{SEC-THM-PROOF}.


\begin{remark}
  \textnormal{ First, observe that $P(x,\infty) \leq \Fa(x),$ for
    every $P$ in the neighborhood set of measures
    $\mathcal{P}_{\alpha,\delta} := \{P : D_\alpha(P,\Pref) \leq
    \delta\}.$ Therefore, for any $\alpha > 1,$ apart from
    characterizing the domain of attraction of $\Fa,$ Theorem
    \ref{EVT-RENYI} offers the following insights on the neighborhood
    $\mathcal{P}_{\alpha,\delta}:$
  \begin{itemize}
  \item[1)] If the reference model belongs to the domain of attraction
    of a Frechet distribution (that is, $\gref >0$), and if $P$ is a
    probability measure that lies in its neighborhood
    $\mathcal{P}_{\alpha,\delta},$ then $P$ must satisfy that
    \begin{equation}\label{tailFrechet}
    P(x,\infty) = O \left(x^{- \frac{\alpha-1}{\alpha \gref} +
          \epsilon}\right),
    \end{equation} 
    as $x \rightarrow \infty,$ for every
    $\epsilon > 0.$ This conclusion is a consequence of
    \eqref{FRECH-CHAR}: $\Fa$ is in the domain of attraction of $G_{\gamma^*}$, then by \eqref{FRECH-CHAR} we have
  \[
    \Fa(x) = L(x) x^{-1/\gamma^*} = L(x) x^{-\frac{\alpha-1}{\alpha\gref}}, 
   \]
     and the observation that
    $P(x,\infty) \leq \Fa(x).$ In addition, as in the proof of Theorem
    \ref{EVT-RENYI}, one can exhibit a measure
    $P \in \mathcal{P}_{\alpha,\delta}$ such that
    $P(x,\infty) \geq c x^{- (\alpha-1) / \alpha\gref}$ for some
    $c > 0$ and all large enough $x$.
  \item[2)] On the other hand, if the reference model belongs to the
    Gumbel domain of attraction ($\gref = 0$), then every
    $P \in \mathcal{P}_{\alpha,\delta}$ satisfies
    $P(x,\infty) = o (x^{-\epsilon}),$ as $x \rightarrow \infty,$ for
    every $\epsilon > 0.$
\item[3)] Now consider the case where
  $\Pref \in \mathcal{D}(G_{\gamma_{\textnormal{ref}}})$ for some
  $\gref < 0$ (that is, the reference model belongs to the domain of
  attraction of a Weibull distribution). Let $x_{_F}^* < \infty$ denote
  the supremum of its bounded support. In that case, any probability
  measure $P$ that belongs to the neighborhood
  $\mathcal{P}_{\alpha,\delta}$ must satisfy that $P(-\infty,x_{_F}^*) = 1$
  and
  \[P(x_{_F}^* - \epsilon, x_{_F}^*) =
    O\left(\epsilon^{-\frac{\alpha-1}{\alpha \gref} - \epsilon'}
    \right),\] as $\epsilon \rightarrow 0,$ for every $\epsilon' > 0.$
  In addition, one can exhibit a measure
  $P \in \mathcal{P}_{\alpha,\delta}$ such that
  $P(x_{_F}^* - \epsilon,x_{_F}^*) \geq c\epsilon^{- (\alpha-1)
    /\alpha \gref},$ for some positive constant $c$ and all
  $\epsilon > 0$ sufficiently small.
\end{itemize}
It is important to remember that the above properties hold for all
$\alpha > 1,$ and is not dependent on $\delta.$}
\label{REM-NEIGH-PROPS}
\end{remark}
For a fixed reference model $\Pref,$ it is evident from Remark
\ref{REM-NEIGH-PROPS} that the neighborhoods
$\mathcal{P}_{\alpha, \delta} = \{P : D_\alpha(P,\Pref) \leq \delta\}$
include probability distributions with heavier and heavier tails as
$\alpha$ approaches 1 from above. This is in line with the observation
that $D_\alpha(P,\Pref)$ is a non-decreasing function in $\alpha,$ and
hence larger neighborhoods $\mathcal{P}_{\alpha, \delta}$ for smaller
values of $\alpha.$ In particular, when $\alpha = 1$ and shape
parameter $\gref = 0,$ the quantity
$\gamma^\ast = \gref \alpha/(\alpha-1) $ defined in Theorem
\ref{EVT-GEV-RENYI} is not well-defined. This corresponds to the set
of plausible measures $\{P:D_1(P,G_0) \leq \delta\}$ defined using
KL-divergence around the reference Gumbel model $G_0$. The following
result describes the tail behaviour of $\Fa$ in this case:
\begin{proposition}\label{propspecial}
  Recall the definition of extreme value distributions $G_\gamma$ in
  \eqref{EV-DISTS}. Let
  $\Fb(x) = \sup\{ P(x,\infty) : D_1(P,G_0) \leq \delta\},$ and
  $F_{1,\delta}(x) = 1 -\bar{F}_{1,\delta}(x).$ Then $F_{1,\delta}$
  belongs to the domain of attraction of $G_1.$
\label{PROP-EVT-KL}
\end{proposition}
\noindent The following result, when contrasted with Remark
\ref{REM-NEIGH-PROPS}, better illustrates the difference between the
cases $\alpha > 1$ and $\alpha = 1.$ 
\begin{proposition}
  Recall the definition of $G_{\gamma}$ as in \eqref{EV-DISTS}. For
  every $\delta > 0,$ one can find a probability measure $P$ in the
  neighborhood
  $\{P : D_1(P,G_{\gamma_{\textnormal{ref}}}) \leq \delta\},$ along
  with positive constants $c_+$ or $c_-$ or $c_0,$ and $x_+$ or
  $x_0$ or $\epsilon_-$ such that
  \begin{itemize}
  \item[a)] $P(x,\infty)  \geq c_+ \log^{-3}x$ for every $x > x_+,$ if
    $\gref > 0;$ 
  \item[b)] $P(x,\infty) \geq c_0 x^{-1}\log^{-2}x$ for every $x > x_0,$ if
    $\gref = 0;$ and 
  \item[c)] $P(-\infty,x_{_G}^*) = 1$ and
    $P(x_{_G}^*-\epsilon,x_{_G}^*) \geq c_3\log^{-3}
    \frac{1}{\epsilon}$ for every $\epsilon < \epsilon_-,$ if
    $\gref < 0.$ Here, the right endpoint
    $x_{_G}^* = \sup\{x: G_{\gref}(x) < 1\}$ is finite because
    $\gref < 0.$
    \label{PROP-KL}
  \end{itemize}
\end{proposition}
\noindent In addition, it is useful to contrast these tail decay
results for neighboring measures with that of the corresponding
reference measure $G_{\gamma_{\textnormal{ref}}}$ characterized in
\eqref{FRECH-CHAR}, \eqref{WEIB-CHAR} or \eqref{GUMB-CHAR}. 
It is evident from this comparison that the worst-case tail
probabilities $\Fa(x)$ decay at a significantly slower rate than the
reference measure when $\alpha = 1$ (the KL-divergence
case). 
Table \ref{TAB-DA-DIFF-CASES} below summarizes the rates of decay of
worst-case tail probabilities $\Fa(\cdot)$ over different choices of
$\alpha$ when the reference model is a GEV distribution. 
Proofs of Theorems \ref{EVT-GEV-RENYI} and
\ref{EVT-RENYI}, Propositions \ref{PROP-EVT-KL} and \ref{PROP-KL} are
presented in Section \ref{SEC-THM-PROOF}.

\begin{table}[h!]
    \caption{A summary of domains of attraction of $F_\alpha(x) =
      1-\Fa(x)$ for GEV models. Throughout the paper, $\gamma^\ast :=
      \frac{\alpha}{\alpha-1}\gref$ \newline} 
    \begin{center}
    \begin{tabular}{ c | c | c }
      &  Domain of attraction of & Domain of
                                                   attraction of\\
      
      Reference model &Worst-case tail $\Fa(\cdot), \ \alpha > 1$&
                                                                   Worst-case
                                                                   tail
                                                                   $\Fa(\cdot), \ \alpha = 1$\\
      &  &  (the KL-divergence case)\\
      &  & \\\hline
      & & \\
      $G_0$ &  $G_0$ & $G_1$\\ 
      (Gumbel light tails) & (Gumbel light tails) & (Frechet heavy tails)\\
      & & \\ \hline
      & & \\
      $G_{\gamma_{\textnormal{ref}}},\gref > 0$& $G_{\gamma^\ast}
                                                     $ & --\\
      (Frechet heavy tails) & (Frechet heavy tails) & (slow logarithmic decay of\\
      & &  $\Fa(x)$ as $x \rightarrow \infty$) \\
      & & \\\hline
      & & \\
      $G_{\gamma_\textnormal{ref}},\  \gref < 0$  & $G_{\gamma^\ast}$ &  -- \\ 
      (Weibull)  & (Weibull) & (slow logarithmic decay of $\Fa(x)$ to 0 \\
      & & at a finite right endpoint $x^*$) 
              \label{TAB-DA-DIFF-CASES}
    \end{tabular}
\end{center}
\end{table}


\section{Robust estimation of VaR}
\label{SEC-ROB-VAR-EST}
Given independent samples $X_1,\ldots,X_N$ from an unknown
distribution $F,$ we consider the problem of estimating
$F^{\leftarrow}(p)$ for values of $p$ close to 1. In this section, we
develop a data-driven algorithm for estimating robust upper bounds for
these extreme quantiles by employing traditional extreme value theory
in tandem with the insights derived in Sections \ref{SEC-NONPARAM-ROB}
and \ref{SEC-ASYMP-CHAR}. Our motivation has been to provide
conservative estimates for $F^{\leftarrow}(p)$ that are robust against
incorrect model assumptions as well as calibration errors. 

Naturally, the first step in the estimation procedure is to arrive at
a reference model $\Pg(-\infty,x) = G_{\gamma_0}(a_0x+b_0)$ for the
distribution of block-maxima
$M_n$
. Once we have a candidate model $\Pg$ for $M_n$, the $p^n$-th
quantile of the distribution $\Pg$ serves as an estimator for
$F^{\leftarrow}(p).$ Instead, if we have a family of candidate models
(as in Sections \ref{SEC-NONPARAM-ROB} and \ref{SEC-ASYMP-CHAR}) for
$M_n,$ a corresponding robust alternative to this estimator is to
compute the worst-case quantile estimate over all the candidate models
as below:
\begin{align}
  \hat{x}_p := \sup \big\{ G^{\leftarrow}(p^n) : D_\alpha(G,\Pg) \leq
  \delta \big\}. 
\label{ROB-EST}
\end{align}
Here $G^{\leftarrow}$ denotes the usual inverse function
$G^{\leftarrow}(u) = \inf\{ x : G(x) \geq u \}$ with respect to
distribution $G.$ Since the framework of Section
\ref{SEC-NONPARAM-ROB} is limited to optimization over objective
functionals in the form of expectations (as in \eqref{OPT-GEN-PROB}),
it is immediately not clear whether the supremum in \eqref{ROB-EST}
can be evaluated using tools developed in Section
\ref{SEC-NONPARAM-ROB}. Therefore, let us proceed with the following
alternative: First, compute the worst-case tail distribution
\begin{align*}
  \Fa(x) := \sup \left\{ G(x,\infty): D_\alpha(G,\Pg) \leq \delta
  \right\}, \quad x \in \mathbb{R}
\end{align*}
over all candidate models, and compute the corresponding inverse
\[F_{\alpha,\delta}^{\leftarrow}(p^n) := \inf\{ x : 1 - \Fa(x) \geq p^n \}.\]
The estimate $\hat{x}_p$ (defined as in \eqref{ROB-EST}) is indeed
equal to $F_{\alpha,\delta}^{\leftarrow}(p^n),$ and this is the content of
Lemma \ref{LEM-INV-QUANT-EQ}.
\begin{lemma}
  \label{LEM-INV-QUANT-EQ}
  For every $u \in (0,1),$
  $F_{\alpha,\delta}^{\leftarrow}(u) = \sup \left\{ G^{\leftarrow}(u) :
    D_\alpha(G,\Pg) \leq \delta \right\}.$
\end{lemma}
\begin{proof}
  For brevity, let
  $\mathcal{P} = \{G : D_\alpha(G,\Pg) \leq \delta\}.$ Then, it
  follows from the definition of $\Fa(\cdot)$ and
  $F_{\alpha,\delta}^{\leftarrow}(\cdot)$ that
   \begin{align*}
     F_{\alpha,\delta}^{\leftarrow}(u) 
     &= \inf \left\{ x: \sup_{G \in \mathcal{P}} G(x,\infty) \leq 1-u
       \right\}\\  
     &= \inf \bigcap_{G \in \mathcal{P}}\bigg \{ x : G(x,\infty) \leq
       1-u \bigg \}\\
     &= \inf \bigcap_{G \in \mathcal{P}} \big [ G^{\leftarrow}(u),
       \infty \big) = \sup_{G \in \mathcal{P}} G^{\leftarrow}(u). 
   \end{align*}
   This completes the proof of Lemma
   \ref{LEM-INV-QUANT-EQ}. \hfill$\Box$
\end{proof}
Now that we know $\hat{x}_p = F_{\alpha,\delta}^{\leftarrow}(p^n)$ is
the desired upper bound, let us recall from Corollary
\ref{COR-TP-SOLN-CHAR} how to evaluate $\Fa(x)$ for any $x$ of
interest. If $\theta_x > 1$ solves
\begin{align*}
  \Pg(x,\infty) \phi_\alpha(\theta_x) + \Pg(-\infty,x)
  \phi_\alpha\left( \frac{1-\theta_x \Pg(x,\infty)}{\Pg(-\infty,x)}
  \right) = \bd, 
\end{align*} 
then $\Fa(x) = \theta_x \Pg(x,\infty).$ Though $\theta_x$ cannot be
obtained in closed-form, given any $x > 0,$ one can numerically solve
for $\theta_x,$ and compute $\Fa(x)$ to a desired level of
precision. On the other hand, given a level $u \in (0,1),$ it is
similarly possible to compute $F_{\alpha,\delta}^{\leftarrow}(u)$ by
solving for $x$ that satisfies $\Pg(x,\infty) < 1-u$ and
\begin{align}
  \Pg(x,\infty) \phi_\alpha\left( \frac{1-u}{\Pg(x,\infty)}\right) +
  \Pg(-\infty,x) \phi_\alpha \left( \frac{u}{\Pg(-\infty,x)} \right) =
  \bd. 
  \label{SOLVING-WORST-QUANTILE}
\end{align}
Therefore, given $\alpha$ and $\delta,$ it is computationally not any
more demanding to evaluate the robust estimates
$F_{\alpha,\delta}^{\leftarrow}(p^n)$ for $F^{\leftarrow}(p)$. 

\subsection{\textbf{On specifying the parameter $\delta.$}}
\label{Sec-Choosing-Delta} For a given choice of paramter
$\alpha \geq 1,$ there are several divergence estimation methods
available in the literature to obtain an estimate
$\hat{\delta} = D_{\alpha}(\hat{P}_{M_n},P_{GEV}),$ where
$\hat{P}_{M_n}$ is the empirical distribution of $M_n$. For our
examples, we use the $k$-nearest neighbor ($k$-NN) algorithm of
\cite{poczos2011estimation} and \cite{wang2009}. See also
\cite{nguyen2009,Wainwright,Gupta_Srivastava} for similar divergence
estimators. These divergence estimation procedures provide an
empirical estimate of the divergence between sample maxima and the
calibrated GEV model $\Pg.$

The specific details of the $k$-NN divergence estimation procedure we
employ from \cite{poczos2011estimation} and \cite{wang2009} are provided in Remark
\ref{Rem-Div-Est} below:

\begin{remark}\label{remark_knn}
  Suppose $M_{n,1},\dots,M_{n,m}$ are independent samples of ${M_n},$
  and $L_1,\dots,$ $L_l$ are samples from $P_{GEV}$. Define
  $\rho_k(i)$ to be the Euclidean distance between
  $M_{n,i}$ and its
  $k$-th nearest neighbour among all
  $M_{n,1},\dots,M_{n,m}$ and similarly
  $\nu_k(i)$ the distance between $M_{n,i}$ and its
  $k$-th nearest neighbour among all $L_1,\dots,L_l$. The
  $k$-NN based density estimators are
\[
 \hat{p}_k(M_{n,i}) = \frac{k/(m-1)}{|B(\rho_k(i))|} \;\; \text{and}\;\; \hat{q}_k(M_{n,i}) = \frac{k/l}{|B(\nu_k(i))|},
\]
where $|B(\rho_k(i))|$ denotes the volume of a ball with radius
$\rho_k(i)$. Then, for a fixed $\alpha,$ the estimator for
$\delta = D_{\alpha}(P_{M_n},P_{GEV})$ is given by
\[
  \hat{\delta} = \frac{1}{\alpha-1}\log\bigg(\frac{1}{m}\sum_{i=1}^m
  \Big(\frac{(m-1)\rho_k(i)}{l\nu_k(i)}\Big)^{1-\alpha}\cdot\frac{\Gamma(k)^2}{\Gamma(k-\alpha+1)\Gamma(k+\alpha-1)}\bigg), 
\]
\label{Rem-Div-Est}
for $\alpha > 1$, where $\Gamma$ denotes the gamma function, and
\[
  \hat{\delta} = \frac{1}{m}\sum_{i=1}^m
  \log\Big(\frac{l\nu_k(i)}{(m-1)\rho_k(i)}\Big),
\]
for $\alpha = 1$.
\end{remark}

For a fixed choice of $\alpha \geq 1$ and desired $p$ close to 1, the
\textsc{Rob-Estimator}$(p,\alpha)$ procedure in Algorithm
\ref{ALGO-ROB-VAR} below provides a summary of the prescribed
estimation procedure.

\begin{algorithm}[b!]
  \caption{ To compute a robust upper bound $\hat{x}_p$ for
    $\textnormal{VaR}_p(X)$ \newline Given: $N$ independent samples
    $X_1,\ldots,X_N$ of $X,$ a level $p$ close to 1, and a fixed
    choice $\alpha \geq 1.$}
    \begin{algorithmic}
      \Procedure{Rob-Estimator}{$p,\alpha$}
      \State Initialize $n < N,$ and let $m =
      \lfloor \frac{N}{n} \rfloor.$\\

      \State Step 1 (Compute block-maxima): Partition $X_1,\ldots,X_N$
      into blocks of size $n,$ and compute the block maxima for each
      block to obtain samples
      $M_{n,1},\ldots,M_{n,m}$ of maxima $M_n.$\\

      \State Step 2 (Calibrate a reference GEV model): Treat the
      samples $M_{n,1},\ldots,M_{n,m}$ as independent samples coming
      from a member of the GEV family and use a parameter estimation
      technique (for example, maximum-likelihood) to estimate the
      parameters $a_{_0},b_{_0}$ and
      $\gamma_{_0},$ along with suitable confidence intervals.\\

      \State Step 3 (Determine the family of candidate models): For
      chosen $\alpha \geq 1,$ determine $\delta$ using a divergence
      estimation procedure (for an example, see Section
      \ref{Sec-Choosing-Delta}). Then the set
      $\{P : D_\alpha(P,\Pg) \leq \delta \}$
      represents the family of candidate models.\\

      \State Step 4 (Compute the $p^n$-th quantile for the reference
      GEV model, and  as well as the worst-case estimate over  all
      candidate models):\\ 
      Solve for $x$ such that
      $G_{\gamma_{_0}}(a_{_0}x + b_{_0}) = p^n,$ and
      let $x_p$ be the corresponding solution.\\
      Solve for $x > x_p$ in \eqref{SOLVING-WORST-QUANTILE} and let
      the solution be $\hat{x}_p.$\\

      \State \textbf{Return} $x_p$ and $\hat{x}_p$
      \EndProcedure
  \end{algorithmic}
\label{ALGO-ROB-VAR} 
\end{algorithm}

\subsection{\textbf{On specifying the parameter $\alpha.$}}
\label{Sec-Choosing-alpha} To input to the estimation procedure
\textsc{Rob-Estimator}$(p,\alpha)$ in Algorithm \ref{ALGO-ROB-VAR},
one can perhaps choose $\alpha$ via one of the three approaches
explained below:
\begin{itemize}
\item[1)] Choose $\alpha$ so that the corresponding
  $\gamma^\ast = \gamma_0 \alpha / (\alpha-1)$ matches with an
  appropriate confidence interval for the estimate $\gamma_0:$ For
  example, if $\gamma_0 >0$ and the confidence interval for
  $\gamma_0,$ estimated from data, is given by
  $(\gamma_0-\epsilon, \gamma_0 + \epsilon),$ then we choose $\alpha$
  satisfying
\begin{align}
  \gamma_0 \frac{\alpha}{\alpha-1} = \gamma_0 + \epsilon. 
  \label{CHOOSING-ALPHA}
\end{align}
See Examples \ref{EG-RAINFALL-DATA} and \ref{EG-SYNTHETIC-DATA} for
demonstrations of choosing $\alpha$ following this approach.

\item[2)] Alternatively, one can choose $\alpha$ based on domain
  knowledge as well: For example, consider the case where one uses
  Gaussian distribution to model returns of a portfolio. In this
  instance, if a financial expert identifies the returns are instead
  heavy-tailed, then one can take $\alpha = 1$ to account for the
  imperfect assumption of Gaussian tails. See Example \ref{EG-NO-DA}
  for a demonstration of choosing $\alpha$ based on this approach.

\item[3)] One can also adopt the following approach that mimicks the
  cross-validation procedure used in machine learning for choosing
  hyperparameters:\\
  Recall that our objective is to estimate $F^{\leftarrow}(p)$ for
  some $p$ close to 1. With this approach, we first estimate
  $F^{\leftarrow}(q)$ as a plug-in estimator from the empirical
  distribution, for some $q < p;$ while it is desirable that $q$ is
  closer to $p,$ care should be taken in the choice that
  $F^{\leftarrow}(q)$ should be estimable from the given $N$ samples
  with high confidence.\\
  Having estimated $F^{\leftarrow}(q)$ directly from the empirical
  distribution, the idea now is to divide the given $N$ samples,
  uniformly at random, into $K$ mini-batches, each of which is
  independently input as samples to the procedure
  \textsc{Rob-Estimator}($q,\alpha$) in Algorithm \ref{ALGO-ROB-VAR}
  to yield $K$ different robust estimates of $F^{\leftarrow}(q)$ for
  an initially chosen value of $\alpha$ (say, $\alpha = 1$).  If the
  mini-batches are of size $N/r,$ then it is reasonable to choose the
  scale-down factor $r$ to be of the same order of magnitude as
  $(1-q)/(1-p).$ The rationale behind this choice is to subject the
  estimation task (that is, to estimate $F^{\leftarrow}(q)$ with $N/r$
  samples) in cross-validation mini-batches to the same level of
  statistical difficulty as in our original task (which is to estimate
  $F^{\leftarrow}(p)$ with $N$ samples).\\
  We repeat the above experiment for small increments of $\alpha$ to
  identify the largest value of $\alpha$ for which the robust
  estimates obtained from the $K$ sub-problems still cover the plug-in
  estimate for $F^{\leftarrow}(q)$ obtained initially from the
  empirical distribution. We utilize this largest value of $\alpha$
  that performs well in the scaled-down sub-problems to be the choice
  of $\alpha$ for robust estimation of $F^{\leftarrow}(p).$\\
  The third approach avoids using the upper end-point of a confidence
  interval of $\gamma$ to pick $\alpha$. Instead it incorporates a
  trade-off between the choice of $\alpha$ and $\delta$. Estimating
  $\delta$ requires the estimation of the R\'{e}nyi divergence, which
  is typically handled by $k$-NN methods as explained in Remark
  \ref{remark_knn}.  Large values of $\alpha$ may be desirable because
  they generate better upper bounds, but since
  $\alpha\rightarrow D_{\alpha}$ is nondecreasing as mentioned in
  Section \ref{SEC-DIV-MEAS}, it also requires large neighborhoods to
  include the true distribution
and hence large values of $\delta$. Further, by Theorem \ref{EVT-RENYI} if the true distribution has heavier tail than the chosen GEV model, then there does exist a threshold of $\alpha$ over which the neighborhoods will not include the true distribution or any other distributions with the same or more tail heaviness than the true distribution, regardless of how large $\delta$ is. Therefore when the chosen $\alpha$ is so large that the true distribution has the tail with an index greater than $\gamma^\ast$, any attempt to estimate such $\delta$ will be unstable and underestimated and causes the failure of coverage for true quantile. 
The above
cross-validation-like procedure incorporates this trade-off and picks a suitable pair $(\alpha, \delta)$.
Example \ref{Petersburg} gives the corresponding
numerical experiments using this approach. 

\end{itemize}

\subsection{\textbf{Numerical examples}}
\begin{example}
  \textnormal{For a demonstration of the ideas introduced, we consider
    the rainfall accumulation data, due to the study of
    \cite{Coles_rainfall}, from a location in south-west England (see
    also \cite{MR1932132} for further extreme value analysis with the
    dataset). Given annual maxima of daily rainfall accumulations over
    a period of 48 years (1914-1962), we attempt to compute, for
    example, the 100-year return level for the daily rainfall data. In
    other words, we aim to estimate the daily rainfall accumulation
    level that is exceeded about only once in 100 years. As a first
    step, we calibrate a GEV model for the annual
    maxima. Maximum-likelihood estimation of parameters results in the
    following values for shape, scale and location parameters:
    $\gamma_0 = 0.1072, \ a_0 = 9.7284$ and $b_0 = 40.7830.$ The
    100-year return level due to this model yields a point estimate
    98.63mm with a standard error of $\pm 17.67$mm (for 95\%
    confidence interval). It is instructive to compare this with the
    corresponding estimate $106.3 \pm 40.7$mm obtained by fitting a
    generalized Pareto distribution (GPD) to the large exceedances
    (see Example 4.4.1 of \cite{MR1932132}). To illustrate our
    methodology, we pick $\alpha = 2,$ as suggested in
    \eqref{CHOOSING-ALPHA}. Next, we obtain $\delta = 0.05$ as an
    empirical estimate of divergence $D_\alpha$ between the data
    points representing annual maxima and the calibrated GEV model
    $\Pg = G_{\gamma_0}(a_0 x + b_0).$ This step is accomplished using
    a simple $k$-nearest neighbor estimator (see
    \cite{poczos2011estimation}).  Consequently, the worst-case
    quantile estimate over all probability measures satisfying
    $D_\alpha(P, \Pg) \leq \delta$ is computed to be
    $F_\alpha^{\leftarrow}(1-1/100) = 132.24$mm. While not being overly
    conservative, this worst-case 100 year return level of 132.44mm
    also acts as an upper bound to estimates obtained due to different
    modelling assumptions (GEV vs GPD assumptions). To demonstrate the
    quality of estimates throughout the tail, we plot the return
    levels for every $1/(1-p)$ years, for values of $p$ close to 1, in
    Figure \ref{FIG-NUM-EG}(a). While the return levels predicted by
    the GEV reference model is plotted in solid line (with the dash-dot
    lines representing 95\% confidence intervals), the dotted curve
    represents the worst-case estimates $F_\alpha^{\leftarrow}(p).$
    The empirical quantiles are drawn in the dashed line. }
  \label{EG-RAINFALL-DATA}
\end{example}

\setcounter{subfigure}{0}
\begin{figure}[h!]
\begin{center}
  \caption{Plots for Examples \ref{EG-RAINFALL-DATA} and
    \ref{EG-SYNTHETIC-DATA}}\label{FIG-NUM-EG} 
  \subfigure[Quantile plots for rainfall data, Eg.
  \ref{EG-RAINFALL-DATA}]{
    \includegraphics[scale=0.25]{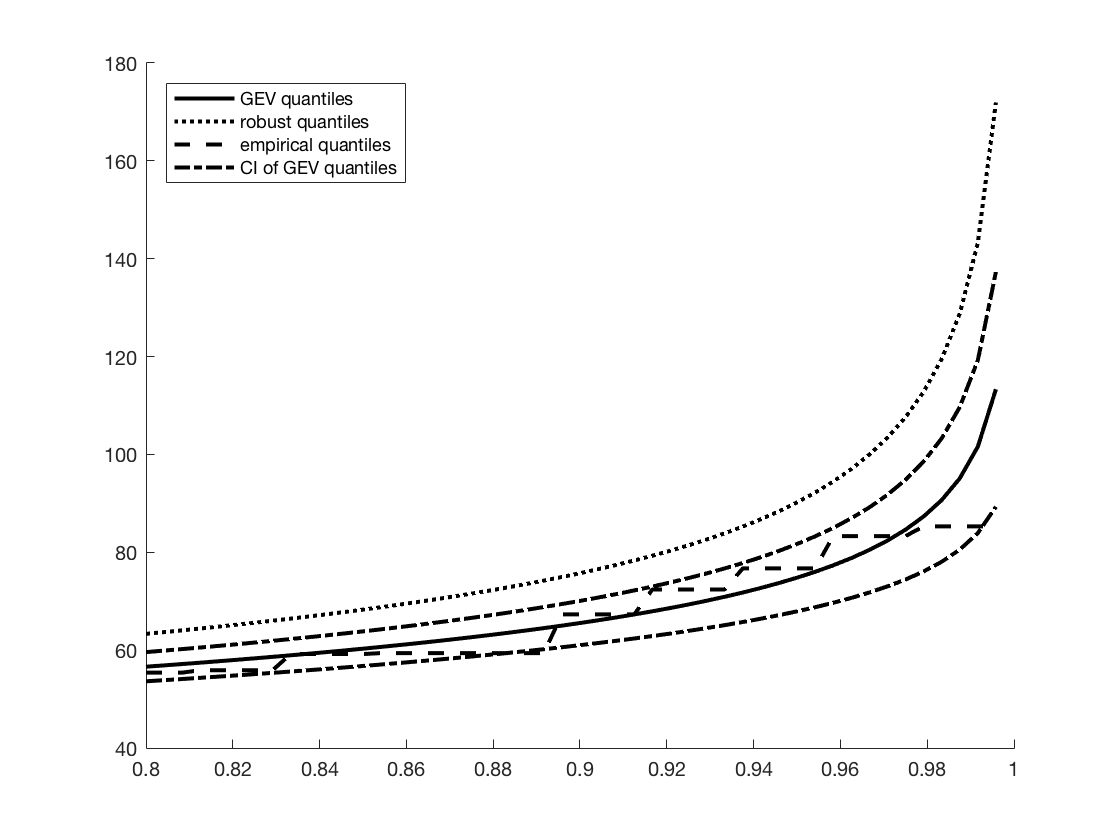}} \hspace{1cm}
  \subfigure[Quantile plots for Pareto data, Eg.
  \ref{EG-SYNTHETIC-DATA}]{
  \includegraphics[ scale = 0.25]{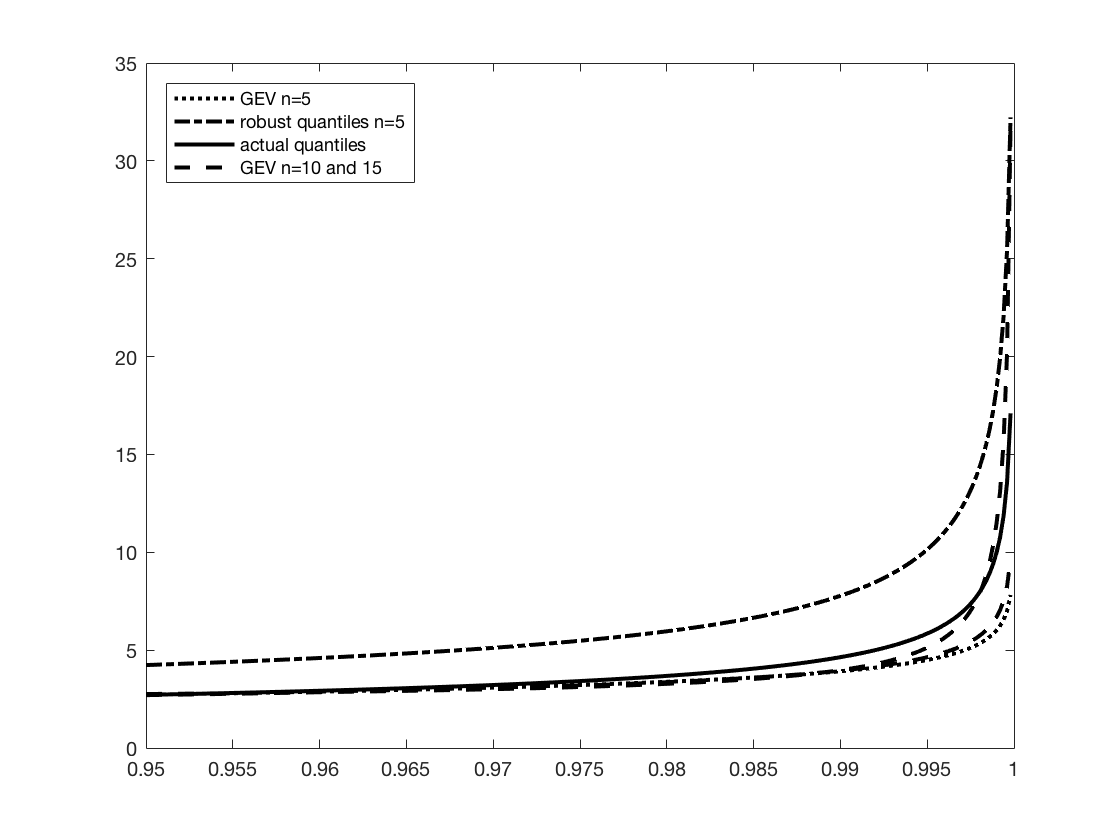}} 
\end{center}
\end{figure}

\begin{example}
  \textnormal{In this example, we are provided with 100 independent
    samples of a Pareto random variable satisfying
    $P\{ X > x\} = 1- F(x) = 1 \wedge x^{-3}.$ As before, the
    objective is to compute quantiles $F^{\leftarrow}(p)$ for values
    of $p$ close to 1. As the entire probability distribution is known
    beforehand, this offers an opportunity to compare the quantile
    estimates returned by our algorithm with the actual quantiles.
    Unlike Example \ref{EG-RAINFALL-DATA}, the data in this example
    does not present a natural means to choose block sizes. As a first
    choice, we choose block size $n = 5$ and perform routine
    computations as in Algorithm \ref{ALGO-ROB-VAR} to obtain a
    reference GEV model $\Pg$ with parameters
    $\gamma_0 = 0.11, a_0 = 0.58, b_0 = 1.88,$ and corresponding
    tolerance parameters $\alpha = 1.5$ and $\delta = 0.8.$ Then the
    worst-case quantile estimate
    $F_\alpha^{\leftarrow}(p^n) = \sup\{ G^{\leftarrow}(p^n) :
    D_\alpha(G, \Pg) \leq \delta\}$
    is immediately calculated for various values of $p$ close to 1,
    and the result is plotted (in the dotted line) against the true quantiles
    $F^{\leftarrow}(p) = (1-p)^{-1/3}$ (in the solid line) in Figure
    \ref{FIG-NUM-EG}(b). These can, in turn, be compared with the
    quantile estimates $x_p$ (in the solid line) due to traditional GEV
    extrapolation with reference model $\Pg.$ Recall that the initial
    choice for block size, $n = 5,$ was arbitrary. One can perhaps
    choose a different block size, which will result in a different
    model for corresponding block-maximum $M_n.$ For example, if we
    choose $n = 10,$ the respective GEV model for $M_{10}$ has
    parameters $\gamma_0 = 0.22, a_0 = 0.55$ and $b_0 = 2.3.$ Whereas,
    if we choose $n = 15,$ the GEV model for $M_{15}$ has parameters
    $\gamma_0 = 0.72, a_0= 0.32 \text{ and } b_0 = 2.66.$ When
    considering the shape parameters, these models are different, and
    subsequently, the corresponding quantile estimates (plotted using
    dashed lines in Figure \ref{FIG-NUM-EG}(b)) are also
    different. However, as it can be inferred from Figure
    \ref{FIG-NUM-EG}(b), the robust quantile estimates (in the dotted line)
    obtained by running Algorithm \ref{ALGO-ROB-VAR} forms a good
    upper bound to the actual quantiles $F^{\leftarrow}(p),$ as well
    as to the quantile estimates due to different GEV extrapolations
    from different block sizes $n = 10$ and 15.}
\label{EG-SYNTHETIC-DATA}
\end{example}

\begin{example}
  \textnormal{The objective of this example is to demonstrate the applicability of
  Algorithm \ref{ALGO-ROB-VAR} in an instance where the traditional
  extrapolation techniques tend to not yield stable estimates. For
  this purpose, we use N = 2000 independent samples of the random
  variable $Y = X + 50 \mathbf{1}(X > 5)$ as input to the maximum
  likelihood based GEV model estimation, with the aim of calculating
  the extreme quantile $F^{\leftarrow}(0.999).$ Here, $F$ denotes the
  distribution function of random variable $Y,$ and $X$ is a Pareto
  random variable with distribution $\max(1-x^{-1.1},0).$ The quantile
  estimates (and the corresponding 95\% confidence intervals) output
  by this traditional GEV estimation procedure, for various choices of
  block sizes, is displayed with the solid line in Figure \ref{FIG-NO-DA}. Even
  for modestly large block size choices, it can be observed that the
  95\% confidence regions obtained from the calibrated GEV models are
  far below the true quantile drawn in the dashed line. This underestimation is
  perhaps because of the sudden shift of samples of block-maxima $M_n$
  from a value less than 5 to a value larger than 55 (recall that the
  distribution $F$ assigns zero probability to the interval
  $(5,55)$).} 

\label{EG-NO-DA}
\end{example}
\begin{figure}[th]
  \begin{center}
    \caption{Plot for Example \ref{EG-NO-DA}, instability in estimated
    quantile $F^{\leftarrow}(0.999)$}
      \includegraphics[scale = 0.25]{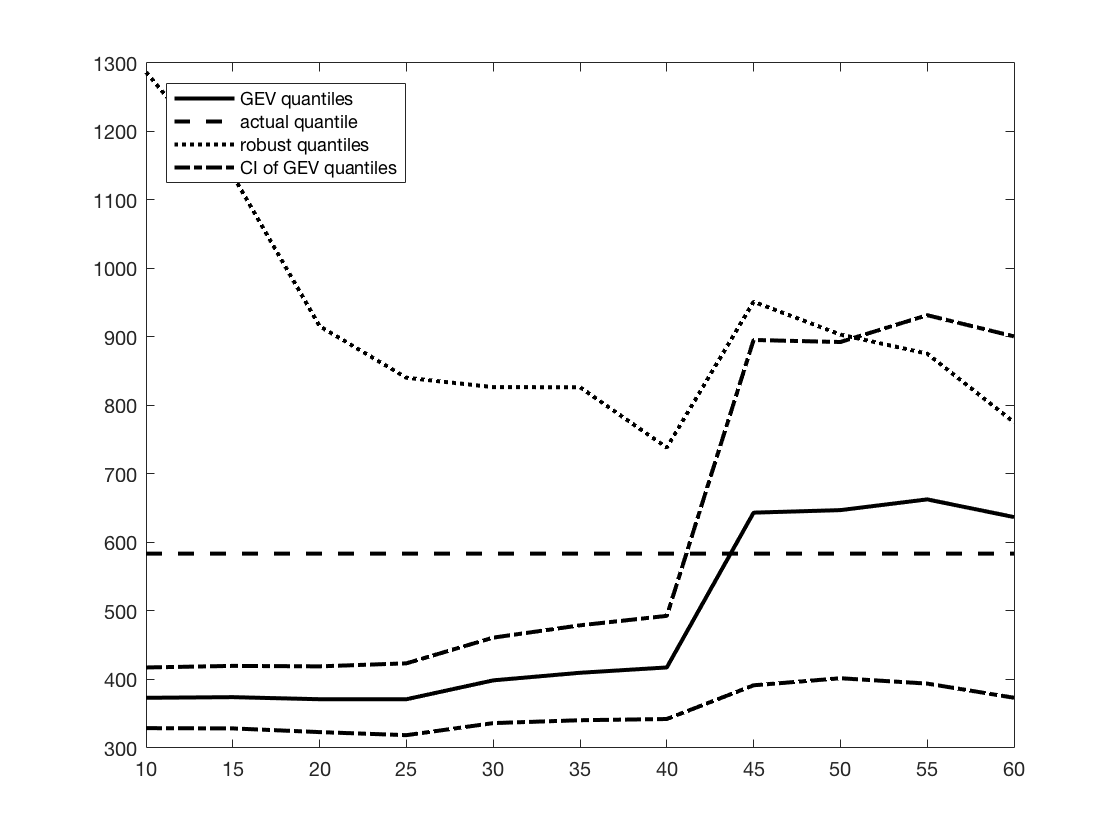}
    \label{FIG-NO-DA} 
  \end{center}
\end{figure}
 
    Next, we use Algorithm \ref{ALGO-ROB-VAR} to yield an upper bound
    that is robust against model errors. Unlike previous examples
    where standard errors are used to calculate the suitable $\alpha,$
    in this example, we use the domain knowledge that the samples of
    $Y$ have finite mean, which means, $\gamma^{\ast} \leq 1.$
    Assuming no additional information, we resort to the conservative
    choice $\gamma^{\ast} = 1.$ The dashed curve in Figure
    \ref{FIG-NO-DA} corresponds to the upper bound on
    $F^{\leftarrow}(0.999)$ output by Algorithm \ref{ALGO-ROB-VAR}. We
    note the following observations: First, the worst case estimates
    output by Algorithm \ref{ALGO-ROB-VAR} indeed act as an upper
    bound for the true quantile (drawn in solid line), irrespective of the
    block-size chosen and the baseline GEV model
    used. 
    Second, for block-sizes smaller than $n = 45,$ it appears that the
    calibrated baseline GEV models are not representative enough of
    the distribution of $M_n,$ and hence higher the value of $\delta$
    for these choices of block sizes. Understandably, this results in
    a conservative worst case estimate for the smaller choices of
    block sizes.  However, we argue that the overall procedure is not
    discouragingly conservative, by observing that the spread of 95\%
    confidence region for block size choices $n = $ 50 to 60 (where
    the traditional GEV calibration appears correct) is comparable to
    the difference between the true quantile and the worst-case
    estimate produced by Algorithm \ref{ALGO-ROB-VAR} for majority of
    block size choices (from $n =$ 20 to 60).

\begin{example}\label{Petersburg}
  \textnormal{In this example we consider the St. Petersburg distribution, which
  is not in the maximum domain of attraction of any GEV distribution
  (see e.g. \cite{fukker2016}). Recall that $X$ is St.Petersburg
  distributed if}
\begin{equation}
  P\{X=2^k\} = 2^{-k},\quad k=1,2,\dots
\end{equation}


\textnormal{Note that the St. Petersburg distribution takes large values with tiny
probability. Let $B$ denote a Bernoulli random variable with parameter
$1/5$. In addition let $W$ be exponentially distributed with mean $8$
and define $Z = B\cdot X+W.$ Suppose we have 5000 data points from the
distribution of $Z$. Similar to the previous example, we want to
estimate its quantile $F^{\leftarrow}(0.999)$. }

\textnormal{Here we demonstrate another approach to choose the parameter
$\alpha$. The idea, as described earlier in Item 3) is to first choose
a tail probability level $q$ for which $F^{\leftarrow}(q)$ can be
accurately estimated from the whole data set. For our example, we take
$q = 0.99$ and compute the plug-in estimate $F^{\leftarrow}(q)$ from
the empirical distribution. Then we independently divide the given
data set uniformly at random into $10$ batches each of size $625$
samples (corresponding to a scale-down factor = 8).  We employ the
procedure \textsc{Rob-Estimator}$(q,\alpha)$ for various values of
$\alpha$ on each of these 10 sub-sampled mini-batches independently,
and choose the largest value of $\alpha$ such that the robust
estimates from each of the 10 sub-samples cover the earlier plug-in
estimate $F^{\leftarrow}(0.99)$. The specific details for this example
are as follows:}

\textnormal{1) The plug-in estimate for $F^{\leftarrow}(0.99)$ from the given 5000
samples is 44.9. Note that with 5000 samples, this estimate from
empirical distribution is with reasonably high confidence. }

\textnormal{2) Resample the data into 10 mini-batches of size 5000/8 = 625
samples. With blocksize = 20 we utilize the procedure
\textsc{Rob-Estimator}$(0.99,\alpha)$ on each of the 10 mini-batches
to choose the largest $\alpha$ such that the respective robust
estimates from all the 10 sub-sampled mini-batches cover the empirical
estimate of $F^{\leftarrow}(0.99)$ obtained from step 1). This
approach leads us to the choice of $\alpha = 4.47$. Computing block maxima
from blocks of samples with size = 48, the subsequent robust upper bound
from the procedure \textsc{Rob-Estimator}$(0.999,4.47)$ turns out to
be $652.90,$ which covers the true quantile,
$F^{\leftarrow}(0.999)= 268.27.$ In contrast, the $95\%$-confidence
interval of GEV estimate is $[93.81, 201.60]$, which fails to cover
the true quantile.}

\textnormal{This approach incorporates the
trade-off between the choice of $\alpha$ and $\delta$.
Large values of $\alpha$ may be
desirable because they generate less conservative upper bounds. But Step 2) avoids picking too large values of $\alpha$, because too large values of $\alpha$, combined with the corresponding estimators for $\delta$ empirically do not lead to good coverage for $F^{\leftarrow}(0.99)$. 
Therefore this cross-validation-like procedure automatically incorporates the trade-off between the choice of hyperparameters $\alpha$ and $\delta$. }
\end{example}

\section{Proofs of main results} 
\label{SEC-THM-PROOF}
\noindent 
In this section, we provide proofs of Theorems \ref{EVT-GEV-RENYI} and
\ref{EVT-RENYI}, along with proofs of Propositions
\ref{Prop-Fa-tail-cdf}, \ref{PROP-EVT-KL} and \ref{PROP-KL}.

\subsection{Proof of Proposition \ref{Prop-Fa-tail-cdf}}
  By definition, $F_{\alpha,\delta}(x)$ is non-decreasing in $x.$
  Since $F_{\alpha,\delta}(x) \leq P_{ref}(-\infty,x),$ we have
  $\lim_{x \rightarrow -\infty}F_{\alpha,\delta}(x) = 0.$ In addition,
  we have from Corollary \ref{COR-TP-SOLN-CHAR} that
  $\bar{F}_{\alpha,\delta}(x) = \theta_x P_{ref}(x, \infty),$ where
  $\theta_x$ satisfies \eqref{THETA-SOLN-EQN}. Since
  $P_{ref}(x,\infty)\phi_{\alpha}(\theta_x) \leq \bd$ (follows from
  \eqref{THETA-SOLN-EQN}), we have
  $\theta_x \leq \phi_{\alpha}^{-1}(\bd/P_{ref}(x,\infty)),$ where
  $\phi_{\alpha}^{-1}(\cdot)$ is the inverse function of
  $\phi_\alpha(\cdot)$ (recall the defintion of $\phi_\alpha(\cdot)$
  in \eqref{DEFN-PHI-BD} to see that the inverse is well-defined for
  every $\alpha \geq 1$). As a result,
\begin{align}
  \Fa(x) \leq
  \phi_{\alpha}^{-1}\left(\frac{\bd}{P_{ref}(x,\infty)}\right)
  P_{ref}(x,\infty).
\label{SIMPLE-UB}
 \end{align}
 If we let $W(x)$ denote the product log function\footnote{$W$ is the
   inverse function of $f(x) = xe^x$}, then
 $\phi^{-1}_\alpha(u) = u^{-1/\alpha}$ when $\alpha > 1$ and
 $\phi^{-1}_\alpha(u) = u/W(u)$ when $\alpha = 1.$ Consequently for
 any $\alpha \geq 1,$
 $\epsilon \phi_{\alpha}^{-1}(1/\epsilon) \rightarrow 0$ as
 $\epsilon \rightarrow 0.$ As a result,
 $\lim_{x \rightarrow \infty}\Fa(x) = 0$ for any choice of
 $\alpha \geq 1$ and $\delta > 0.$ Thus
 $\lim_{x \rightarrow \infty} F_{\alpha,\delta}(x) = 1.$

To show that $F_{\alpha,\delta}(x)$ is right-continuous, we first see
that
\begin{align*}
  F_{\alpha,\delta}(x + \epsilon) - F_{\alpha,\delta} (x) &=
  \sup_{P:D_\alpha(P,P_{ref}) \leq \delta} P(x,\infty) -
  \sup_{P:D_\alpha(P,P_{ref}) \leq \delta} P(x + \epsilon,\infty)\\
  &\leq \sup_{P:D_\alpha(P,P_{ref}) \leq \delta} P(x,x+\epsilon], 
\end{align*}
for any $\epsilon > 0,$ for every choice of $\delta > 0, \alpha \geq
1$ and $P_{ref}.$ Following the same reasoning as in
\eqref{SIMPLE-UB}, we obtain that
\begin{align*}
  \sup_{P:D_\alpha(P,P_{ref}) \leq \delta} P(x,x+\epsilon] \leq
  \phi_{\alpha}^{-1}\left(\frac{\bd}{P_{ref}(x,x+\epsilon]}\right)
  P_{ref}(x,x+\epsilon],
 \end{align*} 
for which the right hand side vanishes when $\epsilon \rightarrow 0.$
As a result, $F_{\alpha,\delta}(x)$ is right-continuous as well, thus
verifying all the properties required to prove that
$F_{\alpha,\delta}(\cdot)$ is a cumulative distribution function. \hfill$\Box$

\subsection{Proofs of Theorems \ref{EVT-GEV-RENYI} - \ref{EVT-RENYI}}
The following technical result, Lemma \ref{lem:h-fn-alpha-ge-1}, is
useful for proving Theorem \ref{EVT-RENYI}. Given $\alpha \geq 1,$
$u \in (0,1)$ and $\bd < \delta_{thr}(u),$ define
\begin{align}
  h(u) := u\theta(u), 
  \label{defn-h}
\end{align}
where $\theta(u)$ is a value of $\theta$ satisfying
\begin{align}
  u\theta^\alpha + (1-u)\left( \frac{1-\theta u}{1-u}\right)^\alpha = \bar{\delta}.
  \label{defn-theta}
\end{align}

\begin{lemma}
  For any $\alpha > 1$ and $\bd > 1,$
  \begin{align*}
    \lim_{u \searrow 0} \frac{\theta(u)}{h^\prime(u)} = \frac{\alpha}{\alpha-1} \quad\text{ and }\quad \lim_{u \searrow 0} \frac{h(u)h^{\prime\prime}(u)}{\left( h^\prime(u)\right)^2} = -\frac{1}{\alpha-1}.
  \end{align*}
  \label{lem:h-fn-alpha-ge-1}
\end{lemma}
\begin{proof}[Proof of Lemma \ref{lem:h-fn-alpha-ge-1}]
  For $u \in (0,1),$ $\theta(u)$ satisfying \eqref{defn-theta}
  exists if $u$ is small enough such that
  $\delta_{thr}(u) := u\phi_\alpha(1/u) \geq \bd$ (see Corollary
  \ref{COR-TP-SOLN-CHAR}). For all such small enough $u,$ an
  application of implicit function theorem gives that,
    \begin{align*}
  \theta^{\prime}(u) = \frac{(1-u)^{\alpha}\theta^{\alpha}(u)-(1-u\theta(u))^{\alpha}+\alpha(1-\theta(u))(1-u\theta(u))^{\alpha-1}}{\alpha(1-u)u[(1-u\theta(u))^{\alpha-1}-(1-u)^{\alpha-1}\theta^{\alpha-1}(u)]},
    \end{align*}
    and consequently, 
\begin{align*}
  h^{\prime}(u) = \frac{(\alpha-1)[(1-u\theta(u))^{\alpha}-(1-u)^{\alpha}\theta^{\alpha}(u)]}{\alpha(1-u)[(1-u\theta(u))^{\alpha-1}-(1-u)^{\alpha-1}\theta^{\alpha-1}(u)]}.
\end{align*}
Since $u\theta(u) \leq u\phi_\alpha^{-1}(\bd/u)$ (see
\eqref{SIMPLE-UB}), we have $u\theta(u) \rightarrow 0$ as
$u \searrow 0.$ Moreover, since
$\theta (u)\geq (\bd - (1-u)^{-(\alpha-1)})/u$ (from \eqref{defn-theta}),
we have that $\theta(u) \rightarrow \infty,$ as $u \searrow 0.$
Combining these observations with the above expression for
$h^\prime(u),$ we arrive at the first conclusion that
$\lim_{u \searrow 0} \theta(u)/h^\prime(u) = \alpha/(\alpha-1).$

To verify the second limiting statement, we proceed by rewriting as
follows: 
\begin{align*}
  \frac{h(u)h^{\prime\prime}(u)}{\left(h^\prime(u)\right)^2}
  =
  \frac{\theta(u)}{h^\prime(u)}\frac{h^{\prime\prime}(u)}{\theta^{\prime}(u)}\frac{u\theta^{\prime}(u)}{h^\prime(u)}
\end{align*}
We know from the above that ${\theta(u)}/{h^\prime(u)}$ converges to
${\alpha}/(\alpha-1)$, as $u \searrow 0;$ and by l'H\^{o}spital's 
rule, we have ${h^{\prime\prime}(u)}/{\theta^{\prime}(u)}$
converges to $(\alpha-1)/\alpha.$ Finally,
\begin{align*}
  \frac{u\theta^{\prime}(u)}{h^\prime(u)}
  &=\frac{u\theta^{\prime}(u)}{\theta(u)+u\theta^{\prime}(u)}\\
  &= \frac{(1-u)^{\alpha}\theta^{\alpha}(u)-(1-u\theta(u))^{\alpha}+\alpha(1-\theta(u))(1-u\theta(u))^{\alpha-1}}{(\alpha-1)[(1-u\theta(u))^{\alpha}-(1-u)^{\alpha}\theta^{\alpha}(u)]},
\end{align*}
which converges to $-\frac{1}{\alpha-1}$, since
$u\theta(u)\rightarrow 0$ as $u\searrow 0$.Combining the above
observations, the verification of the second conclusion that
$ {h(u)h^{\prime\prime}(u)}/{(h^\prime(u))^2} \rightarrow
-1/(\alpha-1)$ is complete.
\hfill$\Box$
\end{proof}

\subsubsection*{Proof of Theorem \ref{EVT-RENYI}}
Our goal is to determine the maximum domain of attraction membership
of $\Fa(x) = \sup\{ P(x,\infty): D_\alpha(P,\Pref) \leq \delta\}.$ For
brevity, let $\bar{F}(x) := \Pref(x,\infty).$ Then for values of $x$
such that $\Pref(x,\infty)$ small enough, we have from Corollary
\ref{COR-TP-SOLN-CHAR} that $\Fa(x) = h(\bar{F}(x)).$ Since
$\bar{F}(\cdot)$ satisfies the regularity conditions in the statement
of Proposition \ref{PROP-EXTREMES-CHAR}, we have
  \begin{align}
    \lim_{x \uparrow x_F^\ast}  \frac{\bar{F}(x)\bar{F}^{\prime\prime}(x)}{\left(\bar{F}^\prime(x)\right)^2} = \gamma_{ref}+1, 
  \label{term-1-dash}
\end{align}
and the following from elementary calculus: 
  \begin{align*}
    \Fa^\prime(x) &= {h}^\prime(\bar{F}(x))\bar{F}^\prime(x) \text{ and} \nonumber \\
    \Fa^{\prime\prime}(x) &= h^{\prime\prime}(\bar{F}(x)) (\bar{F}^\prime(x))^2 + h^\prime(\bar{F}(x))\bar{F}^{\prime\prime}(x).
  \end{align*}
  Combining these observations with the definition in \eqref{defn-h},
  we arrive at,
  \begin{align}
    \frac{\Fa(x) \Fa^{\prime\prime}(x)}{\left(\Fa^\prime(x)\right)^2} = \frac{h(\bar{F}(x))h^{\prime\prime}(\bar{F}(x))}{\left(h^\prime(\bar{F}(x))\right)^2} + \frac{\theta(\bar{F}(x))}{h^\prime(\bar{F}(x))} \left( \frac{\bar{F}(x)\bar{F}^{\prime\prime}(x)}{\left(\bar{F}^\prime(x)\right)^2}\right).
 \label{diff-term}
  \end{align}
  Since $\bar{F}(x) \rightarrow 0$ as $x \rightarrow x_F^\ast,$ it follows
  from Lemma \ref{lem:h-fn-alpha-ge-1}, \eqref{diff-term}
  and \eqref{term-1-dash} that,
  \begin{align*}
    \lim_{x \uparrow x_F^\ast} &\left(\frac{1-F_{\alpha,\delta}}{F_{\alpha,\delta}^\prime}\right)^\prime(x)
    = \lim_{x \uparrow x_F^\ast}  \frac{\Fa(x) \Fa^{\prime\prime}(x)}{\left(\Fa^\prime(x)\right)^2} - 1\\
    &= \lim_{u \searrow 0} \frac{h(u)h^{\prime\prime}(u)}{\left(h^\prime(u)\right)^2} + \lim_{u \searrow 0} \frac{\theta(u)}{h^\prime(u)}\lim_{u \searrow 0}\frac{\bar{F}(x)\bar{F}^{\prime\prime}(x)}{\left(\bar{F}^\prime(x)\right)^2} - 1\\
    &= -\frac{1}{\alpha-1} + \frac{\alpha}{\alpha-1} \left(\gamma_{ref}+1 \right) - 1
    = \frac{\alpha}{\alpha-1} \gamma_{ref}.
  \end{align*}
  Thus, due to the characterization in Proposition
  \ref{PROP-EXTREMES-CHAR}, we have that $F_{\alpha,\delta}$ lies in the maximum
  domain of attraction of $G_{\gamma^\ast}.$ \hfill$\Box$

\subsubsection*{Proof of Theorem \ref{EVT-GEV-RENYI}}
Theorem \ref{EVT-GEV-RENYI} follows as a simple corollary of Theorem
\ref{EVT-RENYI}, once we verify that any GEV model
$G(x) := \Pg(-\infty,x]$ satisfies $G'(x) > 0$ and
$G^{\prime \prime}(x)$ exists in a left neighborhood of
$x_G^* = \sup\{ x: G(x) < 1\},$ along with the property that
\[ \lim_{x \uparrow x_G^*}\left(\frac{1-G}{G'}\right)'(x) = \gref, \]
where $\gref$ is the shape parameter of $G.$ Such a GEV model
satisfies $G(x) = G_{\gamma_{\textnormal{ref}}} (ax + b)$ for some
scaling and translation constants $a$ and $b.$ Therefore, it is enough
to verify these properties only for
$G(x) = G_{\gamma_{\textnormal{ref}}}(x).$ Once we recall the
definition of $G_\gamma$ in \eqref{EV-DISTS}, the desired properties
are elementary exercises in calculus. \hfill$\Box$

\subsection{{Proofs of Propositions \ref{PROP-EVT-KL} -  \ref{PROP-KL}}}
Given $u \in (0,1)$ and $\bd < \delta_{thr}(u),$ let $\theta(u)$ be a
value of $\theta$ that solves the equation,
  \begin{align}
    u\theta\log\theta + (1-\theta u) \log \frac{1-\theta u}{1-u} = \bd.
    \label{defn-theta-eq-1}
  \end{align}
  Define $h(u) := u\theta(u)$ (as in the proof of Theorem
  \ref{EVT-RENYI}, see \eqref{defn-h}).  The following technical result,
  Lemma \ref{lem:h-fn-alpha-eq-1}, is useful for proving Proposition
  \ref{PROP-EVT-KL}.
\begin{lemma}
  For any $\bd > 1,$
  \begin{align*}
    \lim_{u \searrow 0} \frac{\theta(u)}{h^\prime(u)} + \frac{h(u)h^{\prime\prime}(u)}{\left(h^\prime(u)\right)^2} = 2. 
  \end{align*}
  \label{lem:h-fn-alpha-eq-1}
\end{lemma}
\begin{proof}[Proof of Lemma \ref{lem:h-fn-alpha-eq-1}]
    For $u \in (0,1),$ $\theta(u)$ satisfying \eqref{defn-theta}
  exists if $u$ is small enough such that
  $\delta_{thr}(u) := u\phi_\alpha(u) \geq \bd$ (see Corollary
  \ref{COR-TP-SOLN-CHAR}). For all such small enough $u,$ an
  application of implicit function theorem gives that,
  \begin{align*}
  \theta^\prime(u) = \frac{\theta(u) - 1}{u(1-u)L(u)} - \frac{\theta(u)}{u}, \quad \text{ where }   L(u) := \log \frac{\theta(u)(1-u)}{1-u\theta(u)}. 
\end{align*}
\end{proof}
Since $h^\prime(u) = \theta(u) + u\theta^\prime(u),$ it follows that, 
\begin{align}
  \frac{\theta(u) - 1}{h^\prime(u)} = L(u)(1-u).
  \label{inter-prop-eq1-0}
\end{align}
Differentiating both sides and multiplying by $h(u),$ we obtain,
\begin{align*}
  \frac{h(u)\theta^\prime(u)}{h^\prime(u)} - \left( \theta(u) - 1\right) \frac{h(u)h^{\prime\prime}(u)}{\left( h^\prime(u)\right)^2} = u\theta(u)\left( L^\prime(u)(1-u) - L(u)\right). 
\end{align*}
Since the first term in the left hand side above simplifies to
\begin{align*}
  \frac{h(u)\theta^\prime(u)}{h^\prime(u)} = \theta(u) \left( 1 - \frac{1-u}{1-1/\theta(u)}L(u)\right), 
\end{align*}
we obtain that,
\begin{align}
  \left( 1 - \frac{1}{\theta(u)}\right) \frac{h(u)h^{\prime\prime}(u)}{\left( h^\prime(u)\right)^2 } = 1-\frac{(1-u)L(u)}{1-1/\theta(u)} + u L(u) - u(1-u)L^\prime(u).
  \label{inter-prop-eq1-1}
\end{align}
Differentiating $L(u),$ we obtain,
\begin{align*}
  \left( 1- u\theta(u)\right)L^\prime(u) = \frac{\theta^\prime(u)}{\theta(u)} + \frac{\theta(u) - 1}{1-u}. 
\end{align*}
Substituting this observation in \eqref{inter-prop-eq1-1}, we obtain
\begin{align*}
  &\left( 1 - \frac{1}{\theta(u)}\right) \frac{h(u)h^{\prime\prime}(u)}{\left( h^\prime(u)\right)^2 }\\
  &\qquad= 1 + \left(u - \frac{1-u}{1-1/\theta(u)} \right)L(u) - \frac{1-u}{1-u\theta(u)} \left( \frac{u\theta^\prime(u)}{\theta(u)} + \frac{u(\theta(u)-1)}{1-u}\right). 
\end{align*}
Combining this observation with that in \eqref{inter-prop-eq1-0}, we obtain,
\begin{align*}
  &\left( 1 - \frac{1}{\theta(u)}\right)\left(\frac{\theta(u)}{h^\prime(u)} + \frac{h(u)h^{\prime\prime}(u)}{\left( h^\prime(u)\right)^2 }\right)\\
  &\qquad = 1- \frac{1-u\theta(u)}{\theta(u)-1}L(u) - \frac{1-u}{1-u\theta(u)}\left( \frac{1-1/\theta(u)}{(1-u)L(u)} - 1 + \frac{u(\theta(u)-1)}{1-u}\right). 
\end{align*}
Since $u\theta(u) \leq u\phi_1^{-1}(\bd/u)$ (see \eqref{SIMPLE-UB}),
we have $u\theta(u) \rightarrow 0$ as $u \searrow 0.$ Moreover, since
$\theta (u)\geq \phi^{-1}((\bd + \log(1-u))/u)$ (from
\eqref{defn-theta-eq-1}), we have that $\theta(u) \rightarrow \infty,$
as $u \searrow 0.$ Therefore, we have from the above displayed
equation that,
\begin{align*}
  &\lim_{u \searrow 0} 1 \times \left(\frac{\theta(u)}{h^\prime(u)} + \frac{h(u)h^{\prime\prime}(u)}{\left( h^\prime(u)\right)^2 }\right)\\
  &\qquad=  1 -  (1-0)\lim_{u \searrow 0} \frac{L(u)}{\theta(u) - 1} - \frac{1-0}{1-0}\left(\frac{1-0}{1 \times \lim_{u \searrow 0}L(u)} - 1 - \frac{0}{1} \right). 
\end{align*}
It follows from the definition of $L(u)$ that
$L(u) \rightarrow \infty$ as $u \searrow 0;$ due to L'H\^{o}spital's
rule, we also obtain $\lim_{u \searrow 0} L(u)/(\theta(u) - 1) = 0.$
This verifies the statement of Lemma \ref{lem:h-fn-alpha-eq-1}. \hfill$\Box$

\subsubsection*{Proof of Proposition \ref{PROP-EVT-KL}}
Our objective is to identify the maximum domain of attraction
memberiship of the tail probability function,
\[\bar{F}_{1,\delta}(x) := \sup\{ P(x,\infty): D_1(P,G_0) \leq \delta\}.\] For
brevity, let $\bar{G}_0(x) := 1-G_0(x).$ Then for values of $x$ such
that $\bar{G}_0(x)$ small enough, we have from Corollary
\ref{COR-TP-SOLN-CHAR} that $\bar{F}_{1,\delta}(x) = h(\bar{G}_0(x)).$ Since
$\bar{G}_0(\cdot)$ satisfies the regularity conditions in the
statement of Proposition \ref{PROP-EXTREMES-CHAR}, we have
  \begin{align}
    \lim_{x \uparrow \infty}  \frac{\bar{G}_0(x)\bar{G}_0^{\prime\prime}(x)}{\left(\bar{G}_0^\prime(x)\right)^2} = 1, 
  \label{term-2-dash}
  \end{align}
  Then, as in the proof of Theorem \ref{EVT-RENYI}, we have from Lemma
  \ref{lem:h-fn-alpha-eq-1}, \eqref{term-2-dash},\eqref{diff-term}
  that,
  \begin{align*}
    \lim_{x \rightarrow \infty}
    &\left(\frac{1-F_{1,\delta}}{F_{1,\delta}^\prime}\right)^\prime(x)
      = \lim_{x \rightarrow \infty}  \frac{\bar{F}_{1,\delta}(x) \bar{F}_{1,\delta}^{\prime\prime}(x)}{\left(\bar{F}_{1,\delta}^\prime(x)\right)^2} - 1\\
    &= \lim_{x \rightarrow \infty} \frac{h(\bar{G}_0(x))h^{\prime\prime}(\bar{G}_0(x))}{\left(h^\prime(\bar{G}_0(x))\right)^2} +  \frac{\theta(\bar{G}_0(x))}{h^\prime(\bar{G}_0(x))}\frac{\bar{G}_0(x)\bar{G}_0^{\prime\prime}(x)}{\left(\bar{G}_0^\prime(x)\right)^2}-1\\
    &= 2-1 = 1.
  \end{align*}
  Thus, due to the characterization in Proposition
  \ref{PROP-EXTREMES-CHAR}, we have that $F_{1,\delta}$ lies in
  the maximum domain of attraction of $G_{1}.$ \hfill$\Box$

\subsubsection*{Proof of Proposition \ref{PROP-KL}}
First, we treat the case $\gref = 0$: Consider
the probability density function
$f(x) = c(x\log x)^{-2}\mathbf{1}(x \geq 2),$ where $c$ is a
normalizing constant that makes $\int f(x)dx = 1.$ In addition, let
$g(x) = G_0'(x)$ denote the probability density function corresponding
to the distribution $G_0.$ Clearly,
\begin{align*}
  D_1(f,g) &= \int f(x) \log \left( \frac{f(x)}{g(x)}\right)dx\\
              &= c \int_2^\infty (x \log x)^{-2} \log \left( \frac{c(x \log
                x)^{-2}}{\exp(-\exp(-x)) \exp(-x)} \right)dx\\
              &\leq \int_2^\infty \frac{x + \exp(-x) + \log c}{x^2 \log^2x} dx < \infty. 
\end{align*}
Next, consider the family of densities $\{h_a : a \in [0,1]\},$ where
\begin{equation}
  h_a:= af + (1-a)g.
  \label{defn:ha}
\end{equation}

Since $D_1(h_0,g) = 0,$ due to the continuity
of $D_1(h_a, g)$ with respect to $a,$ there exists an
$\bar{a} \in (0,1)$ such that $D_1(h_{\bar{a}}, g) \leq \delta.$ Then, 
\begin{align*}
  \int_x^\infty h_{\bar{a}}(u)du &= \int_x^\infty \left(\bar{a}f +
  (1-\bar{a})g \right)(u)du\\
  &\geq \bar{a}\int_x^\infty
  \frac{c}{u^2\log^2u} du = \frac{\bar{a}c + o(1)}{x \log^2x},
\end{align*}
The asymptotic equivalence used above in the last equality is due to
Karamata's theorem (see Theorem 1 in Chapter VIII.9 of
\cite{MR0210154}). This demonstrates the existence of a probability
distribution $P$ and constants $c_0,x_0$ such that
$P(x,\infty) \geq c_0x^{-1}\log^{-2}x$ for all $x \geq x_0.$

Next, we treat the case $\gref \neq 0$: Consider the probability
measure $Q$ whose Radon-Nikodym derivative is given by,
\begin{align*}
  \frac{dQ}{dG_{\gref}}(x) = \phi_1^{-1}
  \left(\frac{c}{(1-G_{\gref}(x))(1-\log (1-G_{\gref}(x)))^2}\right),  
\end{align*}
for a suitable positive constant $c$. Here $\phi_1^{-1}(\cdot)$
denotes the inverse function of $\phi_1(x).$ Then
$D_1(Q,\Pref) < \infty$ because of the change of variable from $x$ to
$u$ via the relationship  $u = 1-G_{\gref}(x)$ in the integration
below:
\begin{align*}
  \int\phi_1 \left( \frac{d Q}{d G_{\gref}} \right)
  dG_{\gref} = \int_0^1 \frac{c}{u(1-\log u)^2}du < \infty.
\end{align*}

Let $g(x) := G_{\gref}^\prime(x)$ and $f$ denote the probability
density of the measure $Q.$ Consider the family of probability density
functions $\{h_a:a \in [0,1]\},$ where $h_a$ is defined in
$\eqref{defn:ha}.$ Since $D_1(h_0,g) = 0,$ due to the continuity of
$D_1(h_a, g)$ with respect to $a,$ there exists an $\bar{a} \in (0,1)$
such that $D_1(h_{\bar{a}}, g) \leq \delta.$ Moreover, if we let
$A(t) = \phi_1^{-1}(c(1-\log t)^{-2}/t),$ then observe that there
exists a $t_0$ such that $A(t)$ is decreasing in the interval
$(0,t_0).$ Therefore,
\begin{align*}
  \int_x^\infty h_{\bar{a}}(u)du 
  &\geq \bar{a}\int_x^\infty \frac{f(u)}{g(u)}g(u)du = \bar{a}\int_x^\infty A(1-G_{\gref}(u))g(u)du\\
  &\geq \bar{a}A(1-G_{\gref}(x))\int_x^\infty g(u)du = \bar{a}A(1-G_{\gref}(x)) (1-G_{\gref}(x)),
\end{align*}
for all $x$ large enough.  To proceed further, observe that
\[1-G_{\gamma_{\textnormal{ref}}}(x) \geq \bar{c}(1+\gref
  x)^{-1/\gref},\] for some constant $\bar{c} < 1$ and all $x$ close
enough to the right endpoint
$x^*_{_G} := \sup\{ x: G_{\gamma_{\textnormal{ref}}}(x) < 1\}.$ In
addition, $tA(t)$ strictly decreases to 0 as $t$ decreases to
0. Therefore, for all $x$ close to the right endpoint
$x^*_G := \sup\{ x: G_{\gamma_{\textnormal{ref}}}(x) < 1\},$ it
follows that
\begin{align*}
  \int_x^\infty h_{\bar{a}}(u)du \geq A\left(  \bar{c}(1+\gref
  x)^{-1/\gref} \right)  \bar{c}(1+\gref
  x)^{-1/\gref}. 
\end{align*}
Since $\phi_1^{-1}(u) \geq u/\log u$ for large enough $u,$
$A(t) \geq ac{ t^{-1}\left( 1-\log t \right)^{-2} \log^{-1} \left(
    c/t\right)},$
for all $t$ close to 0. As a result, there exists a constant
$c^\prime$ such that $tA(t) \geq c^\prime(1-\log t)^{-3}$ for all $t$
sufficiently close to 0. This allows us to write 
\begin{align*}
  \int_x^\infty h_{\bar{a}}(u)du 
    &\geq c^\prime (1-\log(\bar{c}(1+\gref x)^{-1/\gref}))^{-3}\\
  &=      c^\prime (1+ \log(\bar{c}^{1/\gref} (1+\gref x))/\gref)^{-3},
\end{align*}
for $x$ sufficiently close $x_{_G}^\ast,$ thus verifying the statement
in cases (a) and (c) where $\gref \neq 0.$
This completes the proof of Proposition \ref{PROP-KL}. \hfill$\Box$


\bibliography{Robustness}
\bibliographystyle{abbrv}

\begin{acknowledgements}
  J. Blanchet gratefully acknowledges support from
    NSF, grants 1436700, 1915967, 1820942, 1838676, Norges and as well
    as DARPA award N6. K. Murthy gratefully acknowledges support from
    MOE grant SRG ESD 2018 134.
\end{acknowledgements}

\end{document}